\documentclass[12pt]{article}
\usepackage{mathrsfs}
\usepackage{graphicx}
\usepackage{mathptmx}
\usepackage{latexsym,amsmath,amssymb,amsfonts,amsthm}

\newtheorem{theorem}{Theorem}[section]
\newtheorem{lemma}[theorem]{Lemma}
\newtheorem{example}[theorem]{Example}

\newtheorem{proposition}[theorem]{Proposition}
\newtheorem{remark}[theorem]{Remark}

\newtheorem{definition}[theorem]{Definition}

\newenvironment{proof of theorem 6.2}{{\it Proof of Theorem 6.2}.}{{\hfill $\square$%
    \hskip - \parfillskip}}

\begin{document}
\setcounter{page}{1}
\title{Hyperbolic inverse mean curvature flow}
\author{Zhe Zhou,~~ Chuan-Xi Wu,~~ Jing Mao$^{\ast}$}
\date{}
\protect\footnotetext{\!\!\!\!\!\!\!\!\!\!\!\!\!{$^{\ast}$Corresponding author}
}
\maketitle ~~~\\[-15mm]
\begin{center}
{\footnotesize  Faculty of Mathematics and Statistics,\\
Key Laboratory of Applied Mathematics of Hubei Province,\\
 Hubei University, Wuhan, 430062, China \\
jiner120@163.com, jiner120@tom.com}
\end{center}
\begin{abstract}
In this paper, we prove the short-time existence of hyperbolic
inverse (mean) curvature flow (with or without the specified forcing
term) under the assumption that the initial compact smooth
hypersurface of $\mathbb{R}^{n+1}$ ($n\geqslant2$) is mean convex
and star-shaped. Several interesting examples and some hyperbolic
evolution equations for geometric quantities of the evolving
hypersurfaces have been shown. Besides, under different assumptions
for the initial velocity, we can get the expansion and the
convergence results of a hyperbolic inverse mean curvature flow in
the plane $\mathbb{R}^2$, \emph{whose evolving curves move
normally}.
\end{abstract}

\markright{\sl\hfill Z. Zhou,  C.-X. Wu,  J. Mao \hfill}

\section{Introduction}
\renewcommand{\thesection}{\arabic{section}}
\renewcommand{\theequation}{\thesection.\arabic{equation}}
\setcounter{equation}{0} \setcounter{maintheorem}{0}

Curvature flows is a hot topic in the research of Differential
Geometry in the past several decades. It is well known that Perelman
used the Ricci flow, an intrinsic curvature flow, to successfully
solve the $3$-dimensional Poincar\'{e} conjecture. Among extrinsic
curvature flows, an important one is the mean curvature flow (MCF
fors short), which means a submanifold of a prescribed ambient space
moves with a speed equal to its mean curvature vector. A classical
result in the study of MCF due to Huisken \cite{gh} says that for a
strictly convex, compact hypersurface immersed in $\mathbb{R}^{n+1}$
($n\geqslant2$), if it evolves along the MCF, then evolving
hypersurfaces contract to a single point at some finite time, and
moreover, after area-preserving rescaling, the rescaled evolving
hypersurfaces converge to a round sphere in the
$C^{\infty}$-topology as time tends to infinity. Many improvements
have been obtained after this classical result. Besides, the theory
of MCF also has some interesting applications. For instance, Topping
\cite{pt} used curve shortening flow on surfaces, which is the lower
dimensional version of MCF, to get isoperimetric inequalities on
surfaces. The theory of curve shortening flow  can also be used to
do the image processing (see, e.g., \cite{cf}). The MCF is called
\emph{inward flow}, and conversely, the inverse mean curvature flow
(IMCF for short), which means a submanifold of a prescribed ambient
space moves in direction of the outward unit normal vector of the
submanifold with a speed equal to $1/H$ ($H\neq0$ denotes the mean
curvature), is called \emph{outward flow}. The IMCF is also a very
important extrinsic flow, which has many interesting and important
applications. For instance, the evolution of non-star-shaped initial
surfaces under the IMCF may occur singularities in finite time, but,
through defining a notion of weak solution to IMCF equation,
Huisken-Ilmanen \cite{hi} proved the Riemannian Penrose inequality
by using the method of IMCF (the Riemannian Penrose inequality can
also be proved by applying the positive mass theorem, see \cite{hb}
for details). Using the method of IMCF, Brendle, Hung and Wang
\cite{bhw} proved a sharp Minkowski inequality for mean convex and
star-shaped hypersurfaces in the $n$-dimensional ($n\geqslant3$)
anti-de Sitter-Schwarzschild manifold, which generalized the related
conclusions in the Euclidean space $\mathbb{R}^n$.

The corresponding author, Dr. Jing Mao, has been working on IMCF for
several years and has also obtained some interesting results with
his collaborators. For instance, Chen and Mao \cite{cm} considered
the evolution of a smooth, star-shaped and $F$-admissible ($F$ is a
1-homogeneous function of principle curvatures satisfying some
suitable conditions) embedded closed hypersurface in the
$n$-dimensional ($n\geqslant3$) anti-de Sitter-Schwarzschild
manifold along its outward normal direction has a speed equal to
$1/F$ (clearly, this evolution process is a natural generalization
of IMCF, and we call it \emph{inverse curvature flow}. We write as
ICF for short), and they proved that this ICF exists for all the
time and, after rescaling, the evolving hypersurfaces converge to a
sphere as time tends to infinity. This interesting conclusion has
been improved by Chen, Mao and Zhou \cite{cmz} to the situation that
the ambient space is a warped product $I\times_{\lambda(r)}N^{n}$
with $I$ an unbounded connected interval of $\mathbb{R}$ (i.e., the
set of real numbers) and $N^{n}$ a Riemannian manifold of
nonnegative Ricci curvature. Also for this kind of warped products
$I\times_{\lambda(r)}N^{n}$, under suitable growth assumptions on
the warping function $\lambda(r)$, Chen, Mao, Xiang and Xu
\cite{cmxc} successfully proved that if an $n$-dimensional
($n\geqslant2$) compact $C^{2,\alpha}$-hypersurface with boundary,
which meets a given cone in $I\times_{\lambda(r)}N^{n}$
perpendicularly and is star-shaped with respect to the center of the
cone, evolves along the IMCF, then the flow exists for all the time
and, after rescaling, the evolving hypersurfaces converge to a piece
of the geodesic sphere as time tends to infinity, which generalized
the main conclusion in \cite{Ma1}.

We know that the MCF and the IMCF describe the motion of a
prescribed submanifold, that is, the velocity $\frac{d}{dt}$ equals
some scalar multiple of the unit normal vector of the submanifold.
If the velocity $\frac{d}{dt}$ is replaced by the acceleration
$\frac{d^2}{dt^{2}}$, what happens? Yau \cite{y1} suggested the
following curvature flow
\begin{eqnarray}  \label{h1}
\frac{d^{2}X}{dt^{2}}=H\vec{n},
\end{eqnarray}
where, as before, $H$ denotes the mean curvature and $\vec{n}$ is
the unit inner normal vector of the initial hypersurface
$X(\cdot,0)$, and pointed out very little about the global time
behavior of the evolving hypersurfaces. The curvature flow
(\ref{h1}) can be seen as the hyperbolic version of MCF, and that is
the reason why it is called \emph{hyperbolic mean curvature flow}
(HMCF for short). In fact, if $\mathcal{M}$ is an $n$-dimensional
($n\geqslant2$) smooth compact Riemannnian manifold and $X(\cdot,t)$
is a one-parameter family of smooth hypersurface immersions in
$\mathbb{R}^{n+1}$ satisfying (\ref{h1}), where $X(\cdot,0)$ is the
hypersurface immersion of $\mathcal{M}$ into $\mathbb{R}^{n+1}$,
then it is not hard to show that (\ref{h1}) is a second-order
hyperbolic PDE, which is used to get the short time existence of the
flow (see \cite[Section 2]{hkl} for details). Mao \cite{m1}
considered a hyperbolic curvature flow whose form is given by
(\ref{h1}) plus a forcing term in direction of the position vector,
that is,
\begin{eqnarray*}
\frac{\partial^{2}X}{\partial t^{2}}=H\vec{n}+c(t)X
\end{eqnarray*}
with $c(t)$ a bounded continuous function w.r.t. the time variable
$t$ only, and successfully improved most conclusions in \cite{hkl}
under suitable assumptions.

Based on our research experience on the ICF and the HMCF, it is
natural to consider the the hyperbolic version of the IMCF.

Let $M_0$ be a compact, mean convex, star-shaped smooth hypersurface
of the $(n+1)$-dimensional Euclidean space $\mathbb{R}^{n+1}$
($n\geqslant2$), which is given as an embedding
\begin{eqnarray*}
X_{0}:\mathbb{S}^{n}\rightarrow\mathbb{R}^{n+1},
\end{eqnarray*}
where $\mathbb{S}^{n}\subset\mathbb{R}^{n+1}$ denotes the unit
sphere in $\mathbb{R}^{n+1}$. Define a one-parameter family of
smooth hypersurfaces embedding in $\mathbb{R}^{n+1}$ given by
 \begin{eqnarray*}
 X(\cdot,t):\mathbb{S}^{n}\rightarrow
\mathbb{R}^{n+1}
 \end{eqnarray*}
 with $X(\cdot,0)=X_{0}(\cdot)$, and we say that it is a solution of the \emph{hyperbolic inverse mean
curvature flow} (HIMCF for short) if $ X(\cdot,t)$ satisfies
\begin{eqnarray*}
{\frac{d^{2}}{d
t^{2}}X(x,t)=H^{-1}(x,t)\vec{\nu}(x,t),~~~~\quad\forall x\in
{\mathbb{S}^{n}},~t>0,}
\end{eqnarray*}
where $H(x,t)$ is the mean curvature of $X(x,t)$, $\vec{\nu}(x,t)$
is the unit outward normal vector on $X(x,t)$. If $X(\cdot,0)=X_0$,
$\frac{d X}{d t}(\cdot,0)=X_{1}(x)$ with $X_{1}(x)$ a smooth
vector-valued function on $\mathbb{S}^n$, then one can get the
existence of the one-parameter family of smooth hypersurfaces
$X(\cdot,t)$ embedding in $\mathbb{R}^{n+1}$ on the time interval
$[0,T)$ with $T<\infty$ (see Theorem \ref{maintheorem} for the
precise statement). Besides, under different assumptions for the
initial velocity, we separately discuss the expansion and the
convergence of a HIMCF in the plane $\mathbb{R}^2$, \emph{whose
evolving curves move normally}, in the last section (see Theorem
\ref{main-3} for the precise statement).

\begin{remark}
\rm{ As mentioned before, some interesting conclusions about IMCF or
ICF can be generalized from the setting that the ambient space is
the Euclidean space to the setting of warped products (see, e.g.,
\cite{cm,cmz,cmxc}). Hence, one might ask the following question:
\begin{itemize}

\item \emph{If we consider the HIMCF or the HICF (see Remark \ref{remark2-1} (2) below for this notion) in the warped product $I\times_{\lambda(r)}N^{n}$
with $I$ an unbounded connected interval of $\mathbb{R}$  and
$N^{n}$ a Riemannian manifold of nonnegative Ricci curvature, could
we get similar results to this paper under some suitable assumptions
on $\lambda(r)$?}

\end{itemize}
This question has been solved in \cite{wzm} and the answer is
positive.

}
\end{remark}

\section{Local existence and uniqueness}

\renewcommand{\thesection}{\arabic{section}}
\renewcommand{\theequation}{\thesection.\arabic{equation}}
\setcounter{equation}{0}

In this section, we would like to use the star-shaped assumption to
change the evolution equation in (\ref{IHMCF}) into a single
second-order hyperbolic PDE, which will lead to the short-time
existence and the uniqueness of the flow (\ref{IHMCF}).

Denote by $M_{t}$ the evolving hypersurface under the flow
(\ref{IHMCF}). Since $M_{0}$ is star-shaped, $M_{t}$ also should be
star-shaped on $[0,\epsilon)$ for some small enough $\epsilon>0$ by
continuity. Let the surface $M_{t}$ be represented as a graph over
$\mathbb{S}^{n}$, i.e., the embedding vector $x=(x^{\alpha})$ now
has the components
$$x^{n+1}=u(x,t),~~~~~ x^{i}=x^{i}(t),$$
with $(x^{i})$ local coordinates of $\mathbb{S}^{n}$.
  Furthermore, let $\xi=(\xi^{i})$ be a local coordinate system of
  $M_{t}$, which implies the graphic function $u$ can be written as
$u=u(x(\xi),t)$. Clearly, the outward unit normal vector in $(x,u)$
has the form
\begin{eqnarray*}
\vec{\nu}=\upsilon^{-1}(-D_{i}u,1),
\end{eqnarray*}
where
$$D_{i}u=\frac{\partial u}{\partial x^{i}},$$
$$\upsilon=\left(1+u^{-2}|Du|^{2}\right)^{\frac{1}{2}}=\left(1+u^{-2}\sigma^{ij}D_{i}uD_{j}u\right)^{\frac{1}{2}},$$
$(\sigma_{ij})$ being the metric of $\mathbb{S}^{n}$ in the
coordinates $(x^{i})$ and naturally $(\sigma^{ij})$ being its
inverse. Therefore, now, the Euclidean metric can be written as
$$ds^{2}=dr^{2}+r^{2}\sigma_{ij}dx^{i}dx^{j}.$$

  Then the evolution equation (\ref{IHMCF}) now yields
\begin{eqnarray} \label{kei}
\frac{d^{2}u}{dt^{2}}=\frac{1}{H\upsilon},~~~~\qquad~~\frac{d^{2}x^{i}}{dt^{2}}=-\frac{D^{i}u\cdot
u^{-2}}{H\upsilon}.
\end{eqnarray}
 On the other hand, by the chain rule, we have
\begin{eqnarray*}
\frac{du}{dt}=\frac{\partial u}{\partial
x^{i}}\frac{dx^{i}}{dt}+\frac{\partial u}{\partial t},
\end{eqnarray*}
and
\begin{eqnarray*}
\frac{d^{2}u}{dt^{2}}=\left(\frac{\partial^{2}u}{\partial
x^{i}\partial x^{j}}\frac{dx^{j}}{dt} +\frac{\partial^{2}u}{\partial
x^{i}\partial t}\right)\frac{dx^{i}}{dt}+\frac{\partial u}{\partial
x^{i}}\frac{d^{2}x^{i}}{dt^{2}}+ \frac{\partial^{2}u}{\partial
x^{i}\partial t}\frac{dx^{i}}{dt}+\frac{\partial^{2}u}{\partial
t^{2}}.
\end{eqnarray*}
Substituting (\ref{kei}) into the above equation yields
\begin{eqnarray*}
\frac{\partial^{2}u}{\partial t^{2}}&=&
\frac{d^{2}u}{dt^{2}}-\frac{\partial u}{\partial
x^{i}}\frac{d^{2}x^{i}}{dt^{2}}- \left(\frac{\partial^{2}u}{\partial
x^{i}\partial x^{j}}\frac{dx^{j}}{dt}\frac{dx^{i}}{dt}
+2\frac{\partial^{2}u}{\partial x^{i}\partial t}\frac{dx^{i}}{dt}\right)\\
&=&\frac{1}{H\upsilon}+D_{i}u\cdot\frac{D^{i}u\cdot
u^{-2}}{H\upsilon} -\left(\frac{\partial^{2}u}{\partial
x^{i}\partial x^{j}}\frac{dx^{j}}{dt}\frac{dx^{i}}{dt}
+2\frac{\partial^{2}u}{\partial x^{i}\partial t}\frac{dx^{i}}{dt}\right)\\
&=&\frac{\upsilon}{H}-\left(\frac{\partial^{2}u}{\partial
x^{i}\partial x^{j}}\frac{dx^{j}}{dt}\frac{dx^{i}}{dt}
+2\frac{\partial^{2}u}{\partial x^{i}\partial
t}\frac{dx^{i}}{dt}\right).
\end{eqnarray*}

  Let $\varphi=\log u$. For a graph $M$ over $\mathbb{S}^{n}$, the metric has the
components
\begin{eqnarray*}
g_{ij}=u_{i}u_{j}+u^{2}\sigma_{ij}=u^{2}(\sigma_{ij}+\varphi_{i}\varphi_{j}),
\end{eqnarray*}
and their inverses are
\begin{eqnarray*}
g^{ij}=u^{-2}\left(\sigma^{ij}-\frac{\varphi^{i}\varphi^{j}}{\upsilon^{2}}\right).
\end{eqnarray*}
 Besides, $\upsilon$ can be expressed as
\begin{eqnarray*}
\upsilon=\left(1+u^{-2}\sigma^{ij}D_{i}u
D_{j}u\right)^{\frac{1}{2}}=\left(1+\sigma^{ij}D_{i}\varphi
D_{j}\varphi\right)^\frac{1}{2}
=\left(1+|D\varphi|^{2}\right)^{\frac{1}{2}},
 \end{eqnarray*}
 and the second fundamental
form can be given as the following
\begin{eqnarray*}
h_{ij}&=&-\frac{1}{\upsilon}\left(u_{ij}-u\sigma_{ij}-\frac{2}{u}u_{i}u_{j}\right)\\
&=&\frac{u}{\upsilon}\left(\sigma_{ij}-\frac{u_{ij}}{u}+\frac{2}{u^{2}}u_{i}u_{j}\right)\\
&=&\frac{u}{\upsilon}\left(\sigma_{ij}-\frac{uu_{ij}-u_{i}u_{j}}{u^{2}}+\frac{u_{i}u_{j}}{u^{2}}\right)\\
&=&\frac{u}{\upsilon}\left(\sigma_{ij}-\varphi_{ij}+\varphi_{i}\varphi_{j}\right),
\end{eqnarray*}
Therefore, the mean curvature is
\begin{eqnarray*}
H&=&g^{ij}h_{ij}\\
&=&u^{-2}\left(\sigma_{ij}-\frac{\varphi_{i}\varphi_{j}}{\upsilon^{2}}\right)\cdot
\frac{u}{\upsilon}\left(\sigma_{ij}-\varphi_{ij}+\varphi_{i}\varphi_{j}\right)\\
&=&\frac{1}{u\upsilon}\left(n-\sigma^{ij}\varphi_{ij}+\sigma^{ij}\varphi_{i}\varphi_{j}
-\frac{\sigma_{ij}\varphi^{i}\varphi^{j}}{\upsilon^{2}}+\frac{\varphi^{i}\varphi^{j}}{\upsilon^{2}}\varphi_{ij}
-\frac{\varphi^{i}\varphi^{j}\varphi_{i}\varphi_{j}}{\upsilon^{2}}\right)\\
&=&\frac{1}{u\upsilon}\left(n+\left(-\sigma^{ij}+\frac{\varphi^{i}\varphi^{j}}{\upsilon^{2}}\right)\varphi_{ij}\right),
\end{eqnarray*}
So, together with (\ref{kei}), we can obtain the following equation
\begin{eqnarray} \label{kei2}
\frac{\partial^{2}u}{\partial
t^{2}}=\frac{u\upsilon^{2}}{n+\left(-\sigma^{ij}+\frac{\varphi^{i}\varphi^{j}}{\upsilon^{2}}\right)\varphi_{ij}}
-\left(\frac{\partial^{2}u}{\partial x^{i}\partial
x^{j}}\frac{dx^{i}}{dt}\frac{dx^{j}}{dt}
+2\frac{\partial^{2}u}{\partial x^{i}\partial
t}\frac{dx^{i}}{dt}\right).
\end{eqnarray}

Note that
\begin{eqnarray*}
\frac{\partial \varphi}{\partial t}=\frac{1}{u}\frac{\partial
u}{\partial t},
\end{eqnarray*}
 then together with (\ref{kei2}), we have
\begin{eqnarray} \label{kei3}
\frac{\partial^{2}\varphi}{\partial t^{2}}&=&\frac{1}{u}\frac{\partial^{2}u}{\partial t^{2}}
-\frac{1}{u^{2}}\left(\frac{\partial u}{\partial t}\right)^{2}\nonumber\\
&=&\frac{\upsilon^{2}}{n+\left(-\sigma^{ij}+\frac{\varphi^{i}\varphi^{j}}{\upsilon^{2}}\right)\varphi_{ij}}
-\frac{1}{u}\left(\frac{\partial^{2}u}{\partial x^{i}\partial
x^{j}}\frac{dx^{i}}{dt}\frac{dx^{j}}{dt}
+2\frac{\partial^{2}u}{\partial x^{i}\partial t}\frac{dx^{i}}{dt}\right)-\left(\frac{\partial \varphi}{\partial t}\right)^{2}\nonumber\\
&=&\frac{\upsilon^{2}}{n+\left(-\sigma^{ij}+\frac{\varphi^{i}\varphi^{j}}{\upsilon^{2}}\right)\varphi_{ij}}
-\left[\left(\varphi_{ij}+\varphi_{i}\varphi_{j}\right)\frac{dx^{i}}{dt}\frac{dx^{j}}{dt}
+2\left(\varphi_{it}+\varphi_{i}\varphi_{t}\right)\frac{dx^{i}}{dt}\right]-\left(\frac{\partial
\varphi}{\partial t}\right)^{2}. \qquad
\end{eqnarray}
Let
\begin{eqnarray*}
&&\phi (x, \varphi_{ij}, \varphi_{it}, \varphi_{i}, \varphi_{t},
\varphi):=\frac{\upsilon^{2}}{n+\left(-\sigma^{ij}+\frac{\varphi^{i}\varphi^{j}}{\upsilon^{2}}\right)\varphi_{ij}}
-\nonumber\\
 &&\qquad\qquad \left[\left(\varphi_{ij}+\varphi_{i}\varphi_{j}\right)\frac{dx^{i}}{dt}\frac{dx^{j}}{dt}
+2\left(\varphi_{it}+\varphi_{i}\varphi_{t}\right)\frac{dx^{i}}{dt}\right]-\left(\frac{\partial
\varphi}{\partial t}\right)^{2}.
\end{eqnarray*}
Consider the following equation
\begin{eqnarray}  \label{add-1}
\left\{
\begin{array}{lll}
\frac{\partial^{2}\varphi}{\partial t^{2}}=\phi (x, \varphi_{ij},
\varphi_{it}, \varphi_{i}, \varphi_{t}, \varphi),~~~~\quad\forall
x\in
{\mathbb{S}^{n}},~t>0,\\[2mm]
\frac{\partial\varphi}{\partial t}(\cdot,0)=\varphi_{1}(x),\\[2mm]
 \varphi(\cdot,0)=\varphi_{2}(x), &\quad
\end{array}
\right.
\end{eqnarray}
where $\varphi_{1}(x)$, $\varphi_{2}(x)$ are smooth functions on
$\mathbb{S}^n$.

First, by the standard theory of second-order hyperbolic PDEs, we
have the following conclusion.
\begin{lemma}
Assume that $M_{0}$ given as before (which, of course, is a graph
over $\mathbb{S}^n$) has strictly positive mean curvature $H_{0}\in
C^{\infty}(\mathbb{S}^n)$, and $\varphi_{1}(x)$, $\varphi_{2}(x)$
are given as in (\ref{add-1}). Then the following wave equation
\begin{eqnarray} \label{add-2}
\left\{
\begin{array}{lll}
\frac{\partial^{2}\varphi}{\partial t^{2}}=\Delta
\varphi+\frac{1}{H_{0}},~~~~\quad\forall x\in
{\mathbb{S}^{n}},~t>0,\\[2mm]
\frac{\partial\varphi}{\partial t}(\cdot,0)=\varphi_{1}(x)\\[2mm]
 \varphi(\cdot,0)=\varphi_{2}(x) &\quad
\end{array}
\right.
\end{eqnarray}
has a unique solution $\varphi_{0}\in
C^{\infty}\left(\mathbb{S}^{n}\times[0,T_{1})\right)$ with some
$T_{1}>0$.
\end{lemma}
Next, we want to consider the linearization of (\ref{add-1}) around
$\varphi_{0}$.
\begin{lemma}
Let $\varphi_{0}\in
C^{\infty}\left(\mathbb{S}^{n}\times[0,T_{1})\right)$ be the unique
solution of the wave equation (\ref{add-2}) and $\xi\in
C^{\infty}\left(\mathbb{S}^{n}\times [0,T_{1})\right)$. Then there
exists some $T>0$ such that the linearization of (\ref{add-1})
around $\varphi_{0}$ given by
\begin{eqnarray*}
\left\{
\begin{array}{lll}
L_{\varphi_{0}}\varphi:=\frac{\partial^{2}\varphi}{\partial
t^{2}}-[a^{ij}\varphi_{ij}+b^{it}\varphi_{it}
+c^{i}\varphi_{i}+d^{t}\varphi_{t}+e\varphi]=\xi,~~~~\quad\forall
x\in
{\mathbb{S}^{n}},~t>0,\\[2mm]
\frac{\partial\varphi}{\partial t}(\cdot,0)=\varphi_{1}\\[2mm]
 \varphi(\cdot,0)=\varphi_{2} &\quad
\end{array}
\right.
\end{eqnarray*}
has a unique solution $\varphi\in
C^{\infty}\left(\mathbb{S}^{n}\times[0,T)\right).$
\end{lemma}
\begin{proof}
Let $\varphi_{\varepsilon}:=\varphi_{0}+\varepsilon\varphi$. We
obtain the linearized operator $L_{\varphi_{0}}$ of
$\frac{\partial^{2}}{\partial t^{2}}-\phi$ around $\varphi_{0}$ as
\begin{eqnarray*}
L_{\varphi_{0}}\varphi:&=&\frac{d}{d\varepsilon}\Big{|}_{\varepsilon=0}
\left(\frac{\partial^{2}\varphi_{\varepsilon}}{\partial t^{2}}-
\phi(x,(\varphi_{\varepsilon})_{ij},(\varphi_{\varepsilon})_{it},(\varphi_{\varepsilon})_{i},
(\varphi_{\varepsilon})_{t},\varphi_{\varepsilon})\right)\\
&=&\frac{\partial^{2}\varphi}{\partial
t^{2}}-\Big{(}\frac{\partial\phi}{\partial(\varphi_{\varepsilon})_{ij}}
\frac{d(\varphi_{\varepsilon})_{ij}}{d\varepsilon}+\frac{\partial\phi}{\partial(\varphi_{\varepsilon})_{it}}
\frac{d(\varphi_{\varepsilon})_{it}}{d\varepsilon}\\
&&+\frac{\partial\phi}{\partial(\varphi_{\varepsilon})_{i}}
\frac{d(\varphi_{\varepsilon})_{i}}{d\varepsilon}+\frac{\partial\phi}{\partial(\varphi_{\varepsilon})_{t}}
\frac{d(\varphi_{\varepsilon})_{t}}{d\varepsilon}+\frac{\partial\phi}{\partial(\varphi_{\varepsilon})}
\frac{d(\varphi_{\varepsilon})}{d\varepsilon}\Big{)}|_{\varepsilon=0}\\
&=&\frac{\partial^{2}\varphi}{\partial
t^{2}}-\left(\frac{\partial\phi}{\partial(\varphi_{0})_{ij}}\varphi_{ij}+
\frac{\partial\phi}{\partial(\varphi_{0})_{it}}\varphi_{it}
+\frac{\partial\phi}{\partial(\varphi_{0})_{i}}\varphi_{i}+\frac{\partial\phi}{\partial(\varphi_{0})_{t}}\varphi_{t}
+\frac{\partial\phi}{\partial(\varphi_{0})}\varphi\right).
\end{eqnarray*}
So, we have
$$a^{ij}:=\frac{g^{ij}}{H^{2}}(\cdot,(\varphi_{0})_{ij}, (\varphi_{0})_{it}, (\varphi_{0})_{i}, (\varphi_{0})_{t},
(\varphi_{0}))-\frac{dx^{i}}{dt}\frac{dx^{j}}{dt},$$
$$b^{it}:=-2\frac{dx^{i}}{dt},$$
and the first equation in (\ref{add-1}) has the form
\begin{eqnarray*}
\frac{\partial^{2}\varphi}{\partial
t^{2}}=a^{ij}\varphi_{ij}+b^{it}\varphi_{it}+I(x,\varphi_{i},\varphi_{t},\varphi),
\end{eqnarray*}
 where the last term $I(x,\varphi_{i},\varphi_{t},\varphi)$ depends on $x,\varphi_{i},\varphi_{t},\varphi$. Consider the coefficient matrix of terms involving
second-order derivatives of $\varphi$, and then we have
\begin{gather*}
\begin{pmatrix}
-1 & -\frac{dx^{1}}{dt} & \cdots & -\frac{dx^{n}}{dt}\\
-\frac{dx^{1}}{dt} & \frac{1}{H^{2}}g^{11}-\frac{dx^{1}}{dt}\frac{dx^{1}}{dt} & \cdots & \frac{1}{H^{2}}g^{1n}-\frac{dx^{1}}{dt}\frac{dx^{n}}{dt}\\
\vdots & \vdots & \vdots & \vdots\\
-\frac{dx^{n}}{dt} & \frac{1}{H^{2}}g^{n1}-\frac{dx^{n}}{dt}\frac{dx^{1}}{dt} &\cdots & \frac{1}{H^{2}}g^{nn}-\frac{dx^{n}}{dt}\frac{dx^{n}}{dt}
\end{pmatrix}
\end{gather*}
which, by a suitable linear transformation, becomes
\begin{gather*}
\begin{pmatrix}
-1 & 0 & \cdots & 0\\
0 & \frac{1}{H^{2}}g^{11} & \cdots & \frac{1}{H^{2}}g^{1n}\\
\vdots & \vdots & \vdots & \vdots\\
0 & \frac{1}{H^{2}}g^{n1} &\cdots & \frac{1}{H^{2}}g^{nn}
\end{pmatrix}
.
\end{gather*}
At $t=0$, since $H_{0}$ is strictly positive, thus $L_{\varphi_{0}}$
is uniformly hyperbolic in some small time interval $[0,\ell)$.
Therefore, the theory of second-order linear hyperbolic PDEs yields
the result.
\end{proof}

Therefore, we have the following short-time existence.
\begin{theorem} \label{maintheorem} (Local existence and uniqueness) If the initial hypersurface $M_0$ is a compact, mean convex, star-shaped
smooth hypersurface of $\mathbb{R}^{n+1}$ ($n\geqslant2$), which is
given as an embedding
\begin{eqnarray*}
X_{0}:\mathbb{S}^{n}\rightarrow\mathbb{R}^{n+1},
\end{eqnarray*}
then there exists a constant $T_{\rm{max}}>0$ such that the initial
value problem (IVP for short)
\begin{eqnarray}  \label{IHMCF}
\left\{
\begin{array}{ll}
{\frac{d^{2}}{d
t^{2}}X(x,t)=H^{-1}(x,t)\vec{\nu}(x,t),~~~~\quad\forall x\in
{\mathbb{S}^{n}},~t>0,}\\[2mm]
\frac{d X}{d t}(x,0)=X_{1}(x),\\[2mm]
 X(x,0)=X_{0}(x), &\quad
\end{array}
\right.
\end{eqnarray}
 has a unique smooth solution $X(x, t)$ on $\mathbb{S}^{n}
\times [0,T_{max})$, where $X_{1}(x)$ is a smooth vector-valued
function on $\mathbb{S}^n$.
\end{theorem}

\begin{remark} \label{remark2-1}
\rm{ (1) If the IVP (\ref{IHMCF}) is replaced by
\begin{eqnarray} \label{IHMCF-1}
\left\{
\begin{array}{ll}
{\frac{d^{2}}{d
t^{2}}X(x,t)=H^{-1}(x,t)\vec{\nu}(x,t)+c(t)X(x,t),~~~~\quad\forall
x\in
{\mathbb{S}^{n}},~t>0,}\\[2mm]
\frac{d X}{d t}(x,0)=X_{1}(x),\\[2mm]
 X(x,0)=X_{0}(x), &\quad
\end{array}
\right.
\end{eqnarray}
with $c(t)$ a bounded continuous function w.r.t. to $t$, and other
assumptions are the same to those in Theorem \ref{maintheorem}, then
one can also get the local existence and uniqueness of the forced
HIMCF (\ref{IHMCF-1}) since the first evolution equation in
(\ref{IHMCF-1}) is a second-order hyperbolic PDE by nearly the same
argument in this section. Although we only add a forcing term
$c(t)X(x,t)$ in direction of the position vector, the convergent
situation of (\ref{IHMCF-1}) will be much different from
(\ref{IHMCF}), which can be seen from examples shown in Section 3
and Remark \ref{remark3-1}.\\
(2) Let $F$ be a symmetric, positive, 1-homogeneous function defined
on an open cone $\Gamma$ of $\mathbb{R}^{n}$ with vertex in the
origin, which contains the positive diagonal, i.e., all $n$-tuples
of the form $(\lambda,\cdots,\lambda)$, $\lambda>0$. Assume that
$F\in C^{0}(\overline\Gamma)\cap C^{2}(\Gamma)$ is monotone,
concave, i.e.,
\begin{eqnarray*}
\frac{\partial F}{\partial\lambda^i}>0, \qquad i=1,2,\cdots,n,
~~\rm{in}~\Gamma,\\
\frac{\partial^{2}F}{\partial\lambda^i\partial\lambda^j}\leqslant0,
\qquad\qquad \qquad
\end{eqnarray*}
and that
\begin{eqnarray*}
F=0 ~~\rm{on}~\partial\Gamma.
\end{eqnarray*}
We also use the normalization convention $F(1,\cdots,1)=n+1$. Based
on Gerhardt \cite{cg} on the ICF in $\mathbb{R}^{n+1}$, we can
consider the following IVP
\begin{eqnarray} \label{IHMCF-2}
\left\{
\begin{array}{ll}
{\frac{d^{2}}{d
t^{2}}X(x,t)=F^{-1}(x,t)\vec{\nu}(x,t),~~~~\quad\forall x\in
{\mathbb{S}^{n}},~t>0,}\\[2mm]
\frac{d X}{d t}(x,0)=X_{1}(x),\\[2mm]
 X(x,0)=X_{0}(x), &\quad
\end{array}
\right.
\end{eqnarray}
where $F$ defined on $\Gamma$ is a function of principle curvatures
described as above, and other assumptions are the same to those in
Theorem \ref{maintheorem}. Clearly, the IVP (\ref{IHMCF}) is a
special case of the IVP (\ref{IHMCF-2}), and the first evolution
equation in (\ref{IHMCF-2}) is a hyperbolic version of the ICF,
which leads to the fact that we call it \emph{hyperbolic inverse
curvature flow} (HICF for short). We \textbf{claim} that the
hyperbolic flow (\ref{IHMCF-2}) also has a unique smooth solution
$X(x, t)$ on $\mathbb{S}^{n} \times [0,T_2)$ with some $T_{2}>0$. As
the argument in Section 2, together with the first evolution
equation of (\ref{IHMCF-2}), one can obtain the following evolution
equation
\begin{eqnarray} \label{ev-varphi}
\frac{\partial^{2}\varphi}{\partial t^{2}}=\frac{\upsilon}{uF}
-\left[\left(\varphi_{ij}+\varphi_{i}\varphi_{j}\right)\frac{dx^{i}}{dt}\frac{dx^{j}}{dt}
+2\left(\varphi_{it}+\varphi_{i}\varphi_{t}\right)\frac{dx^{i}}{dt}\right]-\left(\frac{\partial
\varphi}{\partial t}\right)^{2}.
\end{eqnarray}
Denote by $\mathcal{M}(\Gamma)$ the class of all real $(n\times
n)$-matrices whose eigenvalues  belong to $\Gamma$. Then one can
define a function $\mathcal{F}$ on $\mathcal{M}(\Gamma)$  as
 \begin{eqnarray*}
 \mathcal{F}(a^{ij})=F(\lambda^i),
 \end{eqnarray*}
where the $(\lambda^i)$ are eigenvalues of the matrix $(a^{ij})$. It
has been proven in \cite{cns} the monotonicity and concavity of $F$
now take the form
\begin{eqnarray} \label{2-m}
\mathcal{F}_{ij}=\frac{\partial\mathcal{F}}{\partial a^{ij}} \qquad
\rm{is~positive~definite},
\end{eqnarray}
and
\begin{eqnarray} \label{2-c}
\mathcal{F}_{ij,rs}=\frac{\partial^{2}\mathcal{F}}{\partial
a^{ij}\partial a^{rs}} \qquad \rm{is~negative~semidefinite}.
\end{eqnarray}
 Consider the tensor
\begin{eqnarray*}
h^{i}_{j}&=&g^{ik}h_{kj}=
\frac{1}{u\upsilon}\left[\delta^{i}_{j}+\left(-\sigma^{ik}+\frac{\varphi^{i}\varphi^{k}}{\upsilon^{2}}\right)\varphi_{kj}\right].
\end{eqnarray*}
Define the symmetric tensor
\begin{eqnarray*}
\widehat{h}_{ij}=\frac{1}{2}(\widetilde{\sigma}_{ik}h^{k}_{j}+\widetilde{\sigma}_{jk}h^{k}_{i}),
\end{eqnarray*}
where
\begin{eqnarray*}
\widetilde{\sigma}_{ij}=\sigma_{ij}+\varphi_{i}\varphi_{j}.
\end{eqnarray*}
Set
\begin{eqnarray*}
\widetilde{h}_{ij}:=\frac{u}{\upsilon}\widehat{h}_{ij}=\upsilon^{-2}\left(\sigma_{ij}-\varphi_{ij}+\varphi_{i}\varphi_{j}\right),
\end{eqnarray*}
then, together with (\ref{ev-varphi}), we have
\begin{eqnarray*}
\frac{\partial^{2}\varphi}{\partial
t^{2}}=\frac{1}{\mathcal{F}(\widetilde{h}_{ij})}
-\left[\left(\varphi_{ij}+\varphi_{i}\varphi_{j}\right)\frac{dx^{i}}{dt}\frac{dx^{j}}{dt}
+2\left(\varphi_{it}+\varphi_{i}\varphi_{t}\right)\frac{dx^{i}}{dt}\right]-\left(\frac{\partial
\varphi}{\partial t}\right)^{2},
\end{eqnarray*}
where the nonlinearity $\mathcal{F}$ only depends on $D\varphi$ and
$D^{2}\varphi$.

Now, we do the linearization process. Set
 $$Q(\varphi, D\varphi, D^{2}\varphi):=\frac{1}{\mathcal{F}(\widetilde{h}_{ij})}
-\left[\left(\varphi_{ij}+\varphi_{i}\varphi_{j}\right)\frac{dx^{i}}{dt}\frac{dx^{j}}{dt}
+2\left(\varphi_{it}+\varphi_{i}\varphi_{t}\right)\frac{dx^{i}}{dt}\right]-\left(\frac{\partial
\varphi}{\partial t}\right)^{2},$$
 then one can obtain
\begin{eqnarray*}
Q^{ij}&=&\frac{\partial Q}{\partial \varphi_{ij}}\\
&=&-\frac{1}{\mathcal{F}^{2}(\widetilde{h}_{ij})}\frac{\partial
\mathcal{F}}{\partial \widetilde{h}_{ij}}\frac{\partial
\widetilde{h}_{ij}}{\partial \varphi_{ij}}
-\frac{dx^{i}}{dt}\frac{dx^{j}}{dt}\\
&=&\frac{1}{\upsilon^{2}\mathcal{F}^{2}}\frac{\partial F}{\partial
\widetilde{h}_{ij}}-\frac{dx^{i}}{dt}\frac{dx^{j}}{dt}.
\end{eqnarray*}
Therefore, we have
 \begin{eqnarray*}
 \frac{\partial^{2}\varphi}{\partial
t^{2}}=Q^{ij}\varphi_{ij}-2\frac{dx^{i}}{dt}\varphi_{it} +I(x,
\varphi_{i}, \varphi_{t}, \varphi),
 \end{eqnarray*}
where the last term $I(x, \varphi_{i}, \varphi_{t}, \varphi)$ only
depends on $x, \varphi_{i}, \varphi_{t}, \varphi$. The coefficient
matrix of terms involving second-order derivatives of $\varphi$ in
the above evolution equation is
\begin{gather*}
\begin{pmatrix}
-1 & -\frac{dx^{1}}{dt} & \cdots & -\frac{dx^{n}}{dt}\\
-\frac{dx^{1}}{dt} & \frac{1}{\upsilon^{2}\mathcal{F}^{2}}\frac{\partial \mathcal{F}}{\partial \widetilde{h}_{11}}-\frac{dx^{1}}{dt}\frac{dx^{1}}{dt} & \cdots & \frac{1}{\upsilon^{2}\mathcal{F}^{2}}\frac{\partial \mathcal{F}}{\partial \widetilde{h}_{1n}}-\frac{dx^{1}}{dt}\frac{dx^{n}}{dt}\\
\vdots & \vdots & \vdots & \vdots\\
-\frac{dx^{n}}{dt} &
\frac{1}{\upsilon^{2}\mathcal{F}^{2}}\frac{\partial
\mathcal{F}}{\partial
\widetilde{h}_{n1}}-\frac{dx^{n}}{dt}\frac{dx^{1}}{dt} &\cdots &
\frac{1}{\upsilon^{2}\mathcal{F}^{2}}\frac{\partial\mathcal{F}}{\partial
\widetilde{h}_{nn}}-\frac{dx^{n}}{dt}\frac{dx^{n}}{dt}
\end{pmatrix}
\end{gather*}
which, by a suitable linear transformation, becomes
\begin{gather} \label{NM}
\begin{pmatrix}
-1 & 0 & \cdots & 0\\
0 & \frac{1}{\upsilon^{2}\mathcal{F}^{2}}\frac{\partial \mathcal{F}}{\partial \widetilde{h}_{11}} & \cdots & \frac{1}{\upsilon^{2}\mathcal{F}^{2}}\frac{\partial \mathcal{F}}{\partial \widetilde{h}_{1n}}\\
\vdots & \vdots & \vdots & \vdots\\
0 & \frac{1}{\upsilon^{2}\mathcal{F}^{2}}\frac{\partial
\mathcal{F}}{\partial \widetilde{h}_{n1}} &\cdots &
\frac{1}{\upsilon^{2}\mathcal{F}^{2}}\frac{\partial\mathcal{F}}{\partial
\widetilde{h}_{nn}}
\end{pmatrix}
,
\end{gather}
which, by (\ref{2-m}) and (\ref{2-c}), implies that the matrix
(\ref{NM}) is negative definite. So, the equation is a second-order
linear hyperbolic PDE. Our claim follows by the standard theory of
second-order linear hyperbolic PDEs.\\
(3) Although we can also get the short-time existence of the IVP
(\ref{IHMCF-2}), in this paper we mainly discuss the IVP
(\ref{IHMCF}) since if the initial hypersurface $M_0$ is more
special (e.g., sphere, cylinder), the evolution equation of the
flow, which in general is a second-order hyperbolic PDE, degenerates
into a second-order ordinary differential equation (ODE for short)
and then the convergent situation of the evolving hypersurfaces can
be easily known by directly checking the explicit solution to the
ODE (for details, see examples shown in Section 3).

}
\end{remark}

\section{Examples}

\renewcommand{\thesection}{\arabic{section}}
\renewcommand{\theequation}{\thesection.\arabic{equation}}
\setcounter{equation}{0}

In order to possibly understand the convergence of HIMCF
(\ref{IHMCF}) well, we would like to consider the following two
interesting examples in this section.

\begin{example}  \label{ex1}
\rm{ Consider a family of spheres in $\mathbb{R}^3$
\begin{eqnarray*}
X(x,t)=r(t)(\cos\alpha\cos\beta,\cos\alpha\sin\beta,\sin\alpha),
\end{eqnarray*}
where $\alpha\in\left[-\frac{\pi}{2},\frac{\pi}{2}\right]$,
$\beta\in[0,2\pi]$. By straightforward computation, the induced
metric and the second fundamental form are given as follows
\begin{eqnarray*}
g_{11}=r^2, \qquad g_{22}=r^{2}\cos^{2}\alpha, \qquad
g_{12}=g_{21}=0,
\end{eqnarray*}
and
\begin{eqnarray*}
h_{11}=r, \qquad h_{22}=r\cos^{2}\alpha, \qquad h_{12}=h_{21}=0,
\end{eqnarray*}
which implies the mean curvature is
\begin{eqnarray*}
H=g^{ij}h_{ij}=\frac{2}{r}.
\end{eqnarray*}
Besides, the outward unit normal vector of each $X(\cdot,t)$ is
$\vec{v}=(\cos\alpha\cos\beta,\cos\alpha\sin\beta,\sin\alpha)$.
Therefore, in this setting, the HIMCF (\ref{IHMCF}) becomes
\begin{equation}\label{ex11}
\left\{\begin{array}{ll}
r_{tt}=\frac{r}{2},\\[2mm]
r(0)=r_{0}>0,~~~~r_{t}(0)=r_{1},&
 \end{array}\right.
\end{equation}
with
$X_{1}(x)=r_{1}(\cos\alpha\cos\beta,\cos\alpha\sin\beta,\sin\alpha)$
for some constant $r_{1}$. Solving (\ref{ex11}) directly yields
\begin{eqnarray*}
r(t)=\frac{r_{0}+\sqrt{2}r_{1}}{2}e^{\frac{\sqrt{2}}{2}t}+\frac{r_{0}-\sqrt{2}r_{1}}{2}e^{-\frac{\sqrt{2}}{2}t}
\end{eqnarray*}
on $[0,T_{\rm{max}})$ for some $0<T_{\rm{max}}\leqslant\infty$. It
is not difficult to know that
\begin{itemize}
\item if $r_{0}+\sqrt{2}r_{1}>0$, then $T_{\rm{max}}=\infty$ (i.e., the flow exists for all the time).
Moreover, if furthermore, $r_{0}-\sqrt{2}r_{1}\leqslant0$, the
evolving spheres \emph{expand} exponentially to the infinity, and if
furthermore, $r_{0}-\sqrt{2}r_{1}>0$, then the evolving spheres
\emph{converge} first for a while and then \emph{expand}
exponentially to the infinity;

\item if $r_{0}+\sqrt{2}r_{1}=0$, then
$r(t)=\sqrt{2}r_{0}e^{-\frac{\sqrt{2}}{2}t}$, which implies
$T_{\rm{max}}=\infty$ and the evolving spheres \emph{converge} to a
single point as time tends to infinity;

\item if $r_{0}+\sqrt{2}r_{1}<0$, then
$T_{\rm{max}}=\frac{\sqrt{2}}{2}\ln\left(\frac{\sqrt{2}r_{1}-r_{0}}{\sqrt{2}r_{1}+r_{0}}\right)$
and the evolving spheres \emph{converge} to a single point as
$t\rightarrow T_{\rm{max}}$.
\end{itemize}
From the above argument, at least we can get an impression that the
convergent situation of the HIMCF (\ref{IHMCF}) is much complicated
and has close relation with the initial data.
 }
\end{example}

Based on Example \ref{ex1}, one can consider the following
high-dimensional case.

\begin{example} \label{ex2}
\rm{  Consider a family of spheres in $\mathbb{R}^{n+1}$
($n\geqslant2$)
\begin{eqnarray*}
X(x,t)=r(t)(\cos\theta_{1},\sin\theta_{1}\cos\theta_{2},\sin\theta_{1}\sin\theta_{2}\cos\theta_{3},\cdots,\\
\sin\theta_{1}\sin\theta_{2}\sin\theta_{3}\cdots\sin\theta_{n-1}\cos\theta_{n},\sin\theta_{1}\sin\theta_{2}\sin\theta_{3}\cdots\sin\theta_{n-1}\sin\theta_{n}),
\end{eqnarray*}
where $\theta_{1}\in\left[-\frac{\pi}{2},\frac{\pi}{2}\right]$,
$\theta_{\beta}\in[0,2\pi]$ for $\beta=2,3,\cdots,n$. By
straightforward computation, the induced metric and the second
fundamental form are given as follows
\begin{eqnarray*}
g_{11}=r^{2}\cos^{2}\alpha, \qquad
g_{22}=g_{33}=\cdots=g_{nn}=r^{2}, \qquad
g_{ij}=g_{ji}=0~~{\mathrm{for}}~~i\neq j,
\end{eqnarray*}
and
\begin{eqnarray*}
h_{11}=r\cos^{2}\alpha, \qquad h_{22}=h_{33}=\cdots=h_{nn}=r, \qquad
h_{ij}=h_{ji}=0~~{\mathrm{for}}~~i\neq j,
\end{eqnarray*}
which implies the mean curvature is
\begin{eqnarray*}
H=g^{ij}h_{ij}=\frac{n}{r}.
\end{eqnarray*}
Similar to Example \ref{ex1}, in this setting, the HIMCF
(\ref{IHMCF}) becomes
\begin{equation}\label{ex22}
\left\{\begin{array}{ll}
r_{tt}=\frac{r}{n},\\[2mm]
r(0)=r_{0}>0,~~~~r_{t}(0)=r_{1},&
 \end{array}\right.
\end{equation}
for some constant $r_{1}$. Solving (\ref{ex22}) directly yields
\begin{eqnarray*}
r(t)=\frac{r_{0}+\sqrt{n}r_{1}}{2}e^{\frac{\sqrt{n}}{n}t}+\frac{r_{0}-\sqrt{n}r_{1}}{2}e^{-\frac{\sqrt{n}}{n}t}
\end{eqnarray*}
on $[0,T_{\rm{max}})$ for some $0<T_{\rm{max}}\leqslant\infty$, and
then we have
\begin{itemize}
\item if $r_{0}+\sqrt{n}r_{1}>0$, then $T_{\rm{max}}=\infty$ (i.e., the flow exists for all the time).
Moreover, if furthermore, $r_{0}-\sqrt{n}r_{1}\leqslant0$, the
evolving spheres \emph{expand} exponentially to the infinity, and if
furthermore, $r_{0}-\sqrt{n}r_{1}>0$, then the evolving spheres
\emph{converge} first for a while and then \emph{expand}
exponentially to the infinity;

\item if $r_{0}+\sqrt{n}r_{1}=0$, then
$r(t)=\sqrt{n}r_{0}e^{-\frac{\sqrt{n}}{n}t}$, which implies
$T_{\rm{max}}=\infty$ and the evolving spheres \emph{converge} to a
single point as time tends to infinity;

\item if $r_{0}+\sqrt{n}r_{1}<0$, then
$T_{\rm{max}}=\frac{\sqrt{n}}{n}\ln\left(\frac{\sqrt{n}r_{1}-r_{0}}{\sqrt{n}r_{1}+r_{0}}\right)$
and the evolving spheres \emph{converge} to a single point as
$t\rightarrow T_{\rm{max}}$.
\end{itemize}
}
\end{example}

\begin{example} \label{ex3}
\rm{ Now, we would like to consider cylinder solution for the HIMCF
(\ref{IHMCF}) in $\mathbb{R}^3$ which has the following form
\begin{eqnarray*}
X(x,t)=(r(t)\cos\alpha,r(t)\sin\alpha,\rho),
\end{eqnarray*}
where $\alpha\in[0,2\pi]$, $\rho\in[0,\rho_0]$ for some
$\rho_{0}>0$. Clearly, the induced metric and the second fundamental
form can be easily computed as follows
\begin{eqnarray*}
g_{11}=r^2, \qquad g_{22}=r^{2}, \qquad g_{12}=g_{21}=0,
\end{eqnarray*}
and
\begin{eqnarray*}
h_{11}=r, \qquad h_{22}=0, \qquad h_{12}=h_{21}=0,
\end{eqnarray*}
which implies the mean curvature is
\begin{eqnarray*}
H=g^{ij}h_{ij}=\frac{1}{r}.
\end{eqnarray*}
Besides, the outward unit normal vector od each $X(\cdot,t)$ is
$\vec{v}=(\cos\alpha,\sin\alpha,0)$. Therefore, in this setting, the
HIMCF (\ref{IHMCF}) becomes
\begin{equation}\label{ex33}
\left\{\begin{array}{ll}
r_{tt}=r,\\[2mm]
r(0)=r_{0}>0,~~~~r_{t}(0)=r_{1},&
 \end{array}\right.
\end{equation}
with $X_{1}(x)=(r_{1}\cos\alpha,r_{1}\sin\alpha,\rho)$ for some
constant $r_{1}$. Solving (\ref{ex33}) directly yields
\begin{eqnarray*}
r(t)=\frac{r_{0}+r_{1}}{2}e^{t}+\frac{r_{0}-r_{1}}{2}e^{-t}
\end{eqnarray*}
on $[0,T_{\rm{max}})$ for some $0<T_{\rm{max}}\leqslant\infty$. It
is not difficult to know that
\begin{itemize}
\item if $r_{0}+r_{1}>0$, then $T_{\rm{max}}=\infty$ (i.e., the flow exists for all the
time). Moreover, if furthermore, $r_{0}-r_{1}\leqslant0$, the
evolving cylinders \emph{expand} exponentially to the infinity, and
if furthermore, $r_{0}-r_{1}>0$, then the evolving cylinders
\emph{converge} first for a while and then \emph{expand}
exponentially to the infinity;

\item if $r_{0}+r_{1}=0$, then
$r(t)=r_{0}e^{-t}$, which implies $T_{\rm{max}}=\infty$ and the
evolving cylinders \emph{converge} to \emph{a straight line} as time
tends to infinity;

\item if $r_{0}+r_{1}<0$, then
$T_{\rm{max}}=\ln\left(\frac{r_{1}-r_{0}}{r_{0}+r_{1}}\right)$ and
the evolving cylinders \emph{converge} to \emph{a straight line} as
$t\rightarrow T_{\rm{max}}$.
\end{itemize}
}
\end{example}

Of course, as shown in Example \ref{ex2}, one can also consider the
high-dimensional case of Example \ref{ex3}, i.e., the generalized
cylinder solutions to the HIMCF (\ref{IHMCF}). However, through a
simple calculation, one can easily find that, similar to the sphere
case, there is no obvious difference between Example \ref{ex3} and
its high-dimensional version.

\begin{remark} \label{remark3-1}
\rm{ If the HIMCF (\ref{IHMCF}) is replaced by the forced HIMCF
(\ref{IHMCF-1}) in examples shown above, then the convergent
situation will be more complicated. For instance, if the replacement
has been made in Example \ref{ex1}, then (\ref{ex11}) will become
\begin{equation*}
\left\{\begin{array}{ll}
r_{tt}=\frac{r}{2}+c(t)r,\\[2mm]
r(0)=r_{0}>0,~~~~r_{t}(0)=r_{1}.&
 \end{array}\right.
\end{equation*}
Denote the solution to the above IVP by $r(t)$. Since $c(t)$ is
bounded continuous, there exist $c_{-}$, $c^{+}$ such that
$c^{-}\leqslant c(t)\leqslant c^{+}$. Consider the following IVPs
\begin{equation*}
\left\{\begin{array}{ll}
r_{tt}=\frac{r}{2}+c^{-}r,\\[2mm]
r(0)=r_{0}>0,~~~~r_{t}(0)=r_{1},&
 \end{array}\right.
\end{equation*}
and
\begin{equation*}
\left\{\begin{array}{ll}
r_{tt}=\frac{r}{2}+c^{+}r,\\[2mm]
r(0)=r_{0}>0,~~~~r_{t}(0)=r_{1},&
 \end{array}\right.
\end{equation*}
whose solutions are denoted by $r^{-}(t)$ and $r^{+}(t)$
respectively. Clearly, $r^{-}(t)\leqslant r(t)\leqslant r^{+}(t)$.
So, the convergent situation of $r(t)$ deeply depends on that of
$r^{-}(t)$, $r^{+}(t)$ which is not simple. This is because that one
has to discuss the sign of $\left(c^{-}+\frac{1}{2}\right)$,
$\left(c^{+}+\frac{1}{2}\right)$, which leads to the fact that the
convergent situation of $r(t)$ here will be more complicated that of
the one described in Example \ref{ex1}. }
\end{remark}

\section{Evolution equations of some geometric quantities}
\renewcommand{\thesection}{\arabic{section}}
\renewcommand{\theequation}{\thesection.\arabic{equation}}
\setcounter{equation}{0}

  Form the evolution equation for the HIMCF (\ref{IHMCF}), we can derive evolution equations for
some geometric quantities of the hypersurface $X(\cdot, t)$, and
these equations will play an important role in the future study on
the HIMCF.
\begin{lemma} \label{lemma4-1}
Under the HIMCF (\ref{IHMCF}), the following identities hold
$$\Delta h_{ij}=\nabla_{i}\nabla_{j}H+Hh_{il}g^{lm}h_{mj}-|A|^{2}h_{ij},$$
$$\Delta |A|^{2}=2g^{ik}g^{jl}h_{kl}\nabla_{i}\nabla_{j}H+2|\nabla A|^{2}+2H tr(A^{3})-2|A|^{4},$$
where
$$|A|^{2}=g^{ij}g^{kl}h_{ik}h_{jl}, ~~~~~ \quad
tr(A^{3})=g^{ij}g^{kl}g^{mn}h_{ik}h_{lm}h_{nj}.$$
\end{lemma}
The proof of Lemma \ref{lemma4-1} can be found in Zhu \cite{zxp}.

\begin{lemma} \label{lemma4-2}
Under the HIMCF (\ref{IHMCF}), it holds that
\begin{eqnarray*}
\frac{\partial^{2}g_{ij}}{\partial t^{2}}=
2H^{-1}h_{ij}+2e^{2\varphi}(\varphi_{it}+\varphi_{i}\varphi_{t})(\varphi_{jt}+\varphi_{j}\varphi_{t}),
\end{eqnarray*}
where $\varphi$ is defined as in Section 2.
\end{lemma}
\begin{proof}
Denote by $\langle,\rangle$ the standard Euclidean metric in
$\mathbb{R}^{n+1}$ in this section. By a direct computation, we have
\begin{eqnarray*}
\frac{\partial^{2}}{\partial
t^{2}}g_{ij}&=&\frac{\partial^{2}}{\partial t^{2}}
\left\langle\frac{\partial X}{\partial x^{i}}, \frac{\partial
X}{\partial x^{j}}\right\rangle=\frac{\partial}{\partial
t}\left(\left\langle \frac{\partial^{2}X}{\partial x^{i}\partial t},
\frac{\partial X}{\partial x^{j}}\right\rangle
+\left\langle \frac{\partial X}{\partial x^{i}}, \frac{\partial^{2}X}{\partial x^{j}\partial t}\right\rangle\right)\\
&=&\left\langle \frac{\partial^{3}X}{\partial t^{2}\partial x^{i}},
\frac{\partial X}{\partial x^{j}}\right\rangle +\left\langle
\frac{\partial^{2}X}{\partial t\partial x^{i}},
\frac{\partial^{2}X}{\partial t\partial x^{j}}\right\rangle
+\left\langle \frac{\partial^{2}X}{\partial t\partial x^{i}},
\frac{\partial^{2}X}{\partial t\partial x^{j}}\right\rangle
+\left\langle \frac{\partial^{3}X}{\partial t^{2}\partial x^{j}}, \frac{\partial X}{\partial x^{i}}\right\rangle\\
&=&2\left\langle \frac{\partial^{3}X}{\partial t^{2}\partial x^{i}},
\frac{\partial X}{\partial x^{j}}\right\rangle
+2\left\langle \frac{\partial^{2}X}{\partial t\partial x^{i}}, \frac{\partial^{2}X}{\partial t\partial x^{j}}\right\rangle\\
&=&2\left\langle \frac{\partial}{\partial x^{i}}(H^{-1}\vec{\nu}),
\frac{\partial X}{\partial x^{j}}\right\rangle
+2\left\langle \frac{\partial^{2}X}{\partial t\partial x^{i}}, \frac{\partial^{2}X}{\partial t\partial x^{j}}\right\rangle\\
&=&2H^{-1}\left\langle h_{ik}g^{kl}\frac{\partial X}{\partial
x^{l}}, \frac{\partial X}{\partial x^{j}}\right\rangle
+2\left\langle \frac{\partial^{2}X}{\partial t\partial x^{i}}, \frac{\partial^{2}X}{\partial t\partial x^{j}}\right\rangle\\
&=&2H^{-1}h_{ij}+2\left\langle \frac{\partial^{2}X}{\partial t\partial x^{i}}, \frac{\partial^{2}X}{\partial t\partial x^{j}}\right\rangle\\
&=&2H^{-1}h_{ij}+2u_{it}u_{jt}=2H^{-1}h_{ij}+2e^{2\varphi}(\varphi_{it}+\varphi_{i}\varphi_{t})(\varphi_{jt}+\varphi_{j}\varphi_{t}),
\end{eqnarray*}
which completes the proof of Lemma \ref{lemma4-2}.
\end{proof}

\begin{lemma} \label{lemma4-3}
Under the HIMCF (\ref{IHMCF}), we have
\begin{eqnarray*}
\frac{\partial^{2}\vec{\nu}}{\partial
t^{2}}&=&H^{-2}g^{ij}\frac{\partial H}{\partial x^{i}}\frac{\partial
X}{\partial x^{j}}
-\frac{1}{\upsilon}g^{ij}e^{\varphi}(\varphi_{it}+\varphi_{i}\varphi_{t})\frac{\partial^{2} X}{\partial t\partial x^{j}}\\
&&+\frac{1}{\upsilon}g^{ij}g^{kl}e^{3\varphi}
(\varphi_{it}+\varphi_{i}\varphi_{t})(\varphi_{l}\varphi_{jt}+3\varphi_{j}\varphi_{l}\varphi_{t}+2\varphi_{j}\varphi_{lt})\frac{\partial
X}{\partial x^{k}},
\end{eqnarray*}
where $\varphi$ and $\upsilon$ are given as in Section 2.
\end{lemma}
\begin{proof}
First, we have
\begin{eqnarray*}
\frac{\partial\vec{\nu}}{\partial t}
=\left\langle\frac{\partial\overrightarrow{\nu}}{\partial t},
\frac{\partial X}{\partial x^{i}}\right\rangle g^{ij}\frac{\partial
X}{\partial x^{j}} =-\left\langle\vec{\nu},
\frac{\partial^{2}X}{\partial t\partial x^{i}}\right\rangle
g^{ij}\frac{\partial X}{\partial x^{j}}.
\end{eqnarray*}
Then, by a direct computation, it follows that
\begin{eqnarray*}
\frac{\partial^{2}\vec{\nu}}{\partial t^{2}}
&=&-\left\langle\frac{\partial\vec{\nu}}{\partial t},
\frac{\partial^{2}X}{\partial t\partial x^{i}}\right\rangle
g^{ij}\frac{\partial X}{\partial x^{j}} -\left\langle\vec{\nu},
\frac{\partial^{3}X}{\partial t^{2}\partial x^{i}}\right\rangle
g^{ij}\frac{\partial X}{\partial x^{j}} -\left\langle\vec{\nu},
\frac{\partial^{2}X}{\partial t\partial x^{i}}\right\rangle
\frac{\partial g^{ij}}{\partial t}
\frac{\partial X}{\partial x^{j}} - \\
&& \quad \left\langle\vec{\nu}, \frac{\partial^{2}X}{\partial
t\partial x^{i}}\right\rangle g^{ij}
\frac{\partial^{2} X}{\partial t\partial x^{j}}\\
&=&g^{ij}g^{kl}\left\langle\vec{\nu}, \frac{\partial^{2}X}{\partial
t\partial x^{k}}\right\rangle\left\langle\frac{\partial X}{\partial
x^{l}}, \frac{\partial^{2}X}{\partial t\partial
x^{i}}\right\rangle\frac{\partial X}{\partial x^{j}}
-\left\langle\vec{\nu}, \frac{\partial}{\partial x^{i}}(H^{-1}\vec{\nu})\right\rangle g^{ij}\frac{\partial X}{\partial x^{j}}\\
&&+\left\langle\vec{\nu}, \frac{\partial^{2}X}{\partial t\partial
x^{i}}\right\rangle g^{ik}g^{jl}\frac{\partial g_{kl}}{\partial
t}\frac{\partial X}{\partial x^{j}} -\left\langle\vec{\nu},
\frac{\partial^{2}X}{\partial t\partial x^{i}}\right\rangle g^{ij}
\frac{\partial^{2} X}{\partial t\partial x^{j}}\\
&=&H^{-2}g^{ij}\frac{\partial H}{\partial x^{i}}\frac{\partial
X}{\partial x^{j}} -g^{ij}\left\langle\vec{\nu},
\frac{\partial^{2}X}{\partial t\partial x^{i}}\right\rangle
\frac{\partial^{2} X}{\partial t\partial x^{j}}
+g^{ij}g^{kl}\left\langle\vec{\nu}, \frac{\partial^{2}X}{\partial
t\partial
x^{i}}\right\rangle\Big{(}\left\langle\frac{\partial^{2}X}{\partial
t\partial x^{j}}, \frac{\partial X}{\partial x^{l}}\right\rangle
+\\
&& \quad 2\left\langle\frac{\partial X}{\partial x^{j}}, \frac{\partial^{2}X}{\partial t\partial x^{l}}\right\rangle\Big{)}\frac{\partial X}{\partial x^{k}}\\
&=&H^{-2}g^{ij}\frac{\partial H}{\partial x^{i}}\frac{\partial
X}{\partial x^{j}}-\frac{1}{\upsilon}g^{ij}u_{it}\frac{\partial^{2}
X}{\partial t\partial x^{j}}
+\frac{1}{\upsilon}g^{ij}g^{kl}u_{it}(u_{l}u_{jt}+2u_{j}u_{lt})\frac{\partial X}{\partial x^{k}}\\
&=&H^{-2}g^{ij}\frac{\partial H}{\partial x^{i}}\frac{\partial
X}{\partial x^{j}}
-\frac{1}{\upsilon}g^{ij}e^{\varphi}(\varphi_{it}+\varphi_{i}\varphi_{t})\frac{\partial^{2} X}{\partial t\partial x^{j}}\\
&&+\frac{1}{\upsilon}g^{ij}g^{kl}e^{3\varphi}
(\varphi_{it}+\varphi_{i}\varphi_{t})(\varphi_{l}\varphi_{jt}+3\varphi_{j}\varphi_{l}\varphi_{t}+2\varphi_{j}\varphi_{lt})\frac{\partial
X}{\partial x^{k}}
\end{eqnarray*}
which completes the proof of Lemma \ref{lemma4-3}.
\end{proof}

\begin{lemma} \label{lemma4-4}
Under the HIMCF (\ref{IHMCF}), we have
\begin{eqnarray*}
&&\frac{\partial^{2}h_{ij}}{\partial t^{2}}=H^{-2}\Delta
h_{ij}-2H^{-3}\nabla_{i}H\nabla_{j}H+|A|^{2}h_{ij}+\frac{1}{\upsilon^2}g^{kl}h_{ij}e^{2\varphi}(\varphi_{kt}+\varphi_{k}\varphi_{t})
(\varphi_{lt}+\varphi_{l}\varphi_{t})+\\
&&\qquad \qquad
\frac{2}{\upsilon}\frac{\partial\Gamma^{k}_{ij}}{\partial
t}e^{\varphi}(\varphi_{kt}+\varphi_{k}\varphi_{t}),
 \end{eqnarray*}
 where $\varphi$ and $\upsilon$ are defined as in Section 2.
\end{lemma}
\begin{proof}
Since
 \begin{eqnarray*}
 \frac{\partial h_{ij}}{\partial
t}=-\frac{\partial}{\partial t}\left\langle\vec{\nu},
\frac{\partial^{2} X}{\partial x^{i}\partial x^{j}}\right\rangle
=-\left\langle\frac{\partial\vec{\nu}}{\partial t},
\frac{\partial^{2} X}{\partial x^{i}\partial x^{j}}\right\rangle
-\left\langle\vec{\nu}, \frac{\partial^{3} X}{\partial t\partial
x^{i}\partial x^{j}}\right\rangle,
\end{eqnarray*}
 we have
\begin{eqnarray*}
\frac{\partial^{2}h_{ij}}{\partial t^{2}}
&=&-\left\langle\frac{\partial^{2}\vec{\nu}}{\partial t^{2}},
\frac{\partial^{2} X}{\partial x^{i}\partial x^{j}}\right\rangle
-2\left\langle\frac{\partial\vec{\nu}}{\partial t},
\frac{\partial^{3} X}{\partial t\partial x^{i}\partial
x^{j}}\right\rangle
-\left\langle\vec{\nu}, \frac{\partial^{4} X}{\partial t^{2}\partial x^{i}\partial x^{j}}\right\rangle\\
&=&-H^{-2}g^{kl}\left\langle\frac{\partial H}{\partial
x^{k}}\frac{\partial X}{\partial x^{l}}, \frac{\partial^{2}
X}{\partial x^{i}\partial x^{j}}\right\rangle
+g^{kl}\left\langle\vec{\nu}, \frac{\partial^{2}X}{\partial
t\partial x^{k}}\right\rangle
\left\langle\frac{\partial^{2}X}{\partial t\partial x^{l}}, \frac{\partial^{2}X}{\partial x^{i}\partial x^{j}}\right\rangle\\
&&-g^{pq}g^{kl}\left\langle\vec{\nu}, \frac{\partial^{2}X}{\partial
t\partial x^{p}}\right\rangle
\left(\left\langle\frac{\partial^{2}X}{\partial t\partial x^{q}},
\frac{\partial X}{\partial x^{l}}\right\rangle
+2\left\langle\frac{\partial X}{\partial x^{q}},
\frac{\partial^{2}X}{\partial t\partial x^{l}}\right\rangle\right)
\left\langle\frac{\partial X}{\partial x^{k}}, \frac{\partial^{2}X}{\partial x^{i}\partial x^{j}}\right\rangle\\
&&+2g^{kl}\left\langle\vec{\nu}, \frac{\partial^{2}X}{\partial
t\partial x^{k}}\right\rangle \left\langle\frac{\partial X}{\partial
x^{l}}, \frac{\partial^{3} X}{\partial t\partial x^{i}\partial
x^{j}}\right\rangle
-\left\langle\vec{\nu}, \frac{\partial^{2}}{\partial x^{i}\partial x^{j}}\left(H^{-1}\vec{\nu}\right)\right\rangle\\
&=&-H^{-2}\frac{\partial H}{\partial x^{k}}\Gamma^{k}_{ij}
+g^{kl}\left\langle\vec{\nu}, \frac{\partial^{2}X}{\partial
t\partial x^{k}}\right\rangle\Gamma^{m}_{ij}
\left\langle\frac{\partial X}{\partial x^{m}},
\frac{\partial^{2}X}{\partial t\partial x^{l}}\right\rangle
-g^{kl}\left\langle\vec{\nu}, \frac{\partial^{2}X}{\partial
t\partial x^{k}}\right\rangle h_{ij}\times\\
&&\left\langle\vec{\nu}, \frac{\partial^{2}X}{\partial t\partial
x^{l}}\right\rangle-g^{pq}\left\langle\vec{\nu},
\frac{\partial^{2}X}{\partial t\partial x^{p}}\right\rangle
\left(\left\langle\frac{\partial^{2}X}{\partial t\partial x^{q}},
\frac{\partial X}{\partial x^{m}}\right\rangle
+2\left\langle\frac{\partial X}{\partial x^{q}}, \frac{\partial^{2}X}{\partial t\partial x^{m}}\right\rangle\right)\Gamma^{m}_{ij}\\
&&+2g^{kl}\left\langle\vec{\nu}, \frac{\partial^{2}X}{\partial
t\partial x^{k}}\right\rangle\Gamma^{m}_{ij}
\left\langle\frac{\partial X}{\partial x^{l}},
\frac{\partial^{2}X}{\partial t\partial x^{m}}\right\rangle
+2\frac{\partial\Gamma^{k}_{ij}}{\partial t}\left\langle\vec{\nu},
\frac{\partial^{2}X}{\partial t\partial x^{k}}\right\rangle
+2g^{kl}h_{ij}\times\\
&&\left\langle\vec{\nu}, \frac{\partial^{2}X}{\partial t\partial
x^{k}}\right\rangle
\left\langle\vec{\nu}, \frac{\partial^{2}X}{\partial t\partial x^{l}}\right\rangle
-2H^{-3}\nabla_{i}H\nabla_{j}H+H^{-2}\frac{\partial^{2}H}{\partial x^{i}\partial x^{j}}+H^{-1}h_{jk}g^{kl}h_{li}\\
&=&H^{-2}\nabla_{i}\nabla_{j}H-2H^{-3}\nabla_{i}H\nabla_{j}H+H^{-1}h_{jk}g^{kl}h_{li}
+g^{kl}h_{ij}\times\\
&&\qquad\qquad \left\langle\vec{\nu}, \frac{\partial^{2}X}{\partial
t\partial x^{k}}\right\rangle\left \langle\vec{\nu},
\frac{\partial^{2}X}{\partial t\partial x^{l}}\right\rangle
+2\frac{\partial\Gamma^{k}_{ij}}{\partial t}\left\langle\vec{\nu},
\frac{\partial^{2}X}{\partial t\partial x^{k}}\right\rangle.
\end{eqnarray*}
    By Lemma \ref{lemma4-1},
$$\Delta h_{ij}=\nabla_{i}\nabla_{j}H+Hh_{il}g^{lm}h_{mj}-|A|^{2}h_{ij},$$
we can obtain
\begin{eqnarray*}
\frac{\partial^{2}h_{ij}}{\partial t^{2}}&=&H^{-2}\Delta
h_{ij}-2H^{-3}\nabla_{i}H\nabla_{j}H+|A|^{2}h_{ij}
+g^{kl}h_{ij}\left\langle\vec{\nu}, \frac{\partial^{2}X}{\partial
t\partial x^{k}}\right\rangle \left\langle\vec{\nu},
\frac{\partial^{2}X}{\partial t\partial x^{l}}\right\rangle
+\\
&& \quad 2\frac{\partial\Gamma^{k}_{ij}}{\partial t}\left\langle\vec{\nu}, \frac{\partial^{2}X}{\partial t\partial x^{k}}\right\rangle\\
&=&H^{-2}\Delta h_{ij}-2H^{-3}\nabla_{i}H\nabla_{j}H+|A|^{2}h_{ij}+\frac{1}{\upsilon^2}g^{kl}h_{ij}u_{kt}u_{lt}+\frac{2}{\upsilon}\frac{\partial\Gamma^{k}_{ij}}{\partial t}u_{kt}\\
&=&H^{-2}\Delta
h_{ij}-2H^{-3}\nabla_{i}H\nabla_{j}H+|A|^{2}h_{ij}+\frac{1}{\upsilon^2}g^{kl}h_{ij}e^{2\varphi}(\varphi_{kt}+\varphi_{k}\varphi_{t})
(\varphi_{lt}+\varphi_{l}\varphi_{t})+\\
&& \quad \frac{2}{\upsilon}\frac{\partial\Gamma^{k}_{ij}}{\partial
t}e^{\varphi}(\varphi_{kt}+\varphi_{k}\varphi_{t}),
\end{eqnarray*}
which completes the proof of Lemma \ref{lemma4-4}.
\end{proof}

\begin{lemma} \label{lemma4-5}
Under the HIMCF (\ref{IHMCF}), we have
\begin{eqnarray*}
\frac{\partial^{2}H}{\partial t^{2}}&=& H^{-2}\Delta
H-2H^{-3}|\nabla H|^{2}-H^{-1}|A|^{2}
+2g^{ik}g^{jp}g^{lq}h_{ij}\frac{\partial g_{pq}}{\partial
t}\frac{\partial g_{kl}}{\partial t}
-2g^{ik}g^{jl}\frac{\partial g_{kl}}{\partial t}\frac{\partial h_{ij}}{\partial t}\\
&&-\left(2g^{ik}g^{jl}h_{ij}-\frac{1}{\upsilon^2}Hg^{kl}\right)e^{2\varphi}(\varphi_{kt}+\varphi_{k}\varphi_{t})(\varphi_{lt}+\varphi_{l}\varphi_{t})
+\frac{2}{\upsilon}g^{ij}\frac{\partial\Gamma^{k}_{ij}}{\partial
t}e^{\varphi}(\varphi_{kt}+\varphi_{k}\varphi_{t}),
\end{eqnarray*}
 where $\varphi$ and $\upsilon$ are defined as in Section 2.
\end{lemma}
\begin{proof}
Noting that $g^{hm}g_{ml}=\delta^{h}_{l}$,
  we can get
$\frac{\partial g^{ij}}{\partial t}=-g^{ik}g^{jl}\frac{\partial
g_{kl}}{\partial t}$,
 which implies
\begin{eqnarray*}
\frac{\partial^{2}g^{ij}}{\partial t^{2}}&=&-\frac{\partial g^{ik}}{\partial t}g^{jl}\frac{\partial g_{kl}}{\partial t}
-g^{ik}\frac{\partial g^{jl}}{\partial t}\frac{\partial g_{kl}}{\partial t}
-g^{ik}g^{jl}\frac{\partial^{2}g_{kl}}{\partial t^{2}}\\
&=&2g^{ik}g^{jp}g^{lq}\frac{\partial g_{pq}}{\partial t}\frac{\partial g_{kl}}{\partial t}
-g^{ik}g^{jl}\frac{\partial^{2}g_{kl}}{\partial t^{2}}.
\end{eqnarray*}
By a direct calculation, we have
\begin{eqnarray*}
\frac{\partial^{2}H}{\partial t^{2}}&=&\frac{\partial^{2}}{\partial t^{2}}(g^{ij}h_{ij})\\
&=&\frac{\partial^{2}g^{ij}}{\partial t^{2}}h_{ij}+2\frac{\partial
g^{ij}}{\partial t}\frac{\partial h_{ij}}{\partial t}
+g^{ij}\frac{\partial^{2}h_{ij}}{\partial t^{2}}\\
&=&\left(2g^{ik}g^{jp}g^{lq}\frac{\partial g^{pq}}{\partial
t}\frac{\partial g^{kl}}{\partial t}
-g^{ik}g^{jl}\frac{\partial^{2}g_{kl}}{\partial t^{2}}\right)h_{ij}
-2g^{ik}g^{jl}\frac{\partial g_{kl}}{\partial t}\frac{\partial h_{ij}}{\partial t}+g^{ij}\frac{\partial^{2}h_{ij}}{\partial t^{2}}\\
&=&2g^{ik}g^{jp}g^{lq}h_{ij}\frac{\partial g_{pq}}{\partial
t}\frac{\partial g_{kl}}{\partial t} -2g^{ik}g^{jl}\frac{\partial
g_{kl}}{\partial t}\frac{\partial h_{ij}}{\partial t}
-g^{ik}g^{jl}h_{ij}
\left(2H^{-1}h_{kl}+2\left\langle \frac{\partial^{2}X}{\partial t\partial x^{k}}, \frac{\partial^{2}X}{\partial t\partial x^{l}}\right\rangle\right)\\
&&+g^{ij}\Big{(}H^{-2}\nabla_{i}\nabla_{j}H-2H^{-3}\nabla_{i}H\nabla_{j}H+H^{-1}h_{jk}g^{kl}h_{li}\\
&&+g^{kl}h_{ij}\left\langle\vec{\nu}, \frac{\partial^{2}X}{\partial
t\partial x^{k}}\right\rangle \left\langle\vec{\nu},
\frac{\partial^{2}X}{\partial t\partial x^{l}}\right\rangle
+2\frac{\partial\Gamma^{k}_{ij}}{\partial t}\left\langle\vec{\nu}, \frac{\partial^{2}X}{\partial t\partial x^{k}}\right\rangle\Big{)}\\
&=&H^{-2}\Delta H-2H^{-3}|\nabla
H|^{2}-H^{-1}|A|^{2}-2g^{ik}g^{jl}h_{ij} \left\langle
\frac{\partial^{2}X}{\partial t\partial x^{k}},
\frac{\partial^{2}X}{\partial t\partial x^{l}}\right\rangle
+\\
&&Hg^{kl}\left\langle\vec{\nu}, \frac{\partial^{2}X}{\partial
t\partial x^{k}}\right\rangle \left\langle\vec{\nu},
\frac{\partial^{2}X}{\partial t\partial
x^{l}}\right\rangle+2g^{ij}\frac{\partial\Gamma^{k}_{ij}}{\partial
t}\left\langle\vec{\nu}, \frac{\partial^{2}X}{\partial
t\partial x^{k}}\right\rangle\\
&& +2g^{ik}g^{jp}g^{lq}h_{ij}\frac{\partial g_{pq}}{\partial
t}\frac{\partial g_{kl}}{\partial t}
-2g^{ik}g^{jl}\frac{\partial g_{kl}}{\partial t}\frac{\partial h_{ij}}{\partial t}\\
&=&H^{-2}\Delta H-2H^{-3}|\nabla
H|^{2}-H^{-1}|A|^{2}-2g^{ik}g^{jl}h_{ij}u_{kt}u_{lt}+\frac{1}{\upsilon^2}Hg^{kl}u_{kt}u_{lt}
+\frac{2}{\upsilon}g^{ij}\frac{\partial\Gamma^{k}_{ij}}{\partial t}u_{kt}\\
&&+2g^{ik}g^{jp}g^{lq}h_{ij}\frac{\partial g_{pq}}{\partial
t}\frac{\partial g_{kl}}{\partial t}
-2g^{ik}g^{jl}\frac{\partial g_{kl}}{\partial t}\frac{\partial h_{ij}}{\partial t}\\
&=&H^{-2}\Delta H-2H^{-3}|\nabla H|^{2}-H^{-1}|A|^{2}
+2g^{ik}g^{jp}g^{lq}h_{ij}\frac{\partial g_{pq}}{\partial
t}\frac{\partial g_{kl}}{\partial t}
-2g^{ik}g^{jl}\frac{\partial g_{kl}}{\partial t}\frac{\partial h_{ij}}{\partial t}\\
&&-\left(2g^{ik}g^{jl}h_{ij}-\frac{1}{\upsilon^2}Hg^{kl}\right)e^{2\varphi}(\varphi_{kt}+\varphi_{k}\varphi_{t})(\varphi_{lt}+\varphi_{l}\varphi_{t})
+\frac{2}{\upsilon}g^{ij}\frac{\partial\Gamma^{k}_{ij}}{\partial
t}e^{\varphi}(\varphi_{kt}+\varphi_{k}\varphi_{t}),
\end{eqnarray*}
which completes the proof of Lemma \ref{lemma4-5}.
\end{proof}

\begin{lemma} \label{lemma4-6}
Under the HIMCF (\ref{IHMCF}), we have
\begin{eqnarray*}
\frac{\partial^{2}}{\partial t^{2}}|A|^{2}&=&H^{-2}\Delta(|A|^{2})-2H^{-2}|\nabla A|^{2}+2H^{-2}|A|^{4}-4H^{-2}|\nabla H|^{2}-4H^{-1}tr(A^{3})\\
&&+2g^{ij}g^{kl}\frac{\partial h_{ik}}{\partial t}\frac{\partial
h_{jl}}{\partial t}
-8g^{im}g^{jn}g^{kl}h_{jl}\frac{\partial g_{mn}}{\partial t}\frac{\partial h_{ik}}{\partial t}\\
&&+2g^{im}h_{ik}h_{jl}\frac{\partial g_{pq}}{\partial t}\frac{\partial g_{mn}}{\partial t}(2g^{jp}g^{nq}g^{kl}+g^{jn}g^{kp}g^{lq})\\
&&+\frac{2}{\upsilon^2}|A|^{2}g^{pq}e^{2\varphi}(\varphi_{pt}+\varphi_{p}\varphi_{t})(\varphi_{qt}+\varphi_{q}\varphi_{t})
+\frac{4}{\upsilon}g^{ij}g^{kl}h_{jl}\frac{\partial\Gamma^{p}_{ik}}{\partial t}e^{\varphi}(\varphi_{pt}+\varphi_{p}\varphi_{t})\\
&&-4g^{im}g^{jn}g^{kl}h_{ik}h_{jl}e^{2\varphi}(\varphi_{mt}+\varphi_{m}\varphi_{t})(\varphi_{nt}+\varphi_{n}\varphi_{t}),
\end{eqnarray*}
 where $\varphi$ is defined as in Section 2.
\end{lemma}

\begin{proof}
By a direct calculation, we have
\begin{eqnarray*}
\frac{\partial^{2}}{\partial
t^{2}}|A|^{2}&=&2g^{kl}h_{ik}h_{jl}\frac{\partial^{2}g^{ij}}{\partial
t^{2}} +2\frac{\partial g^{ij}}{\partial t}\frac{\partial
g^{kl}}{\partial t}h_{ik}h_{jl} +8g^{kl}h_{jl}\frac{\partial
g^{ij}}{\partial t}\frac{\partial h_{ik}}{\partial t} \\
&&\quad +2g^{ij}g^{kl}\frac{\partial h_{ik}}{\partial
t}\frac{\partial h_{jl}}{\partial t}
+2g^{ij}g^{kl}h_{jl}\frac{\partial^{2}h_{ik}}{\partial t^{2}}\\
&=&2g^{kl}h_{ik}h_{jl}\left(2g^{im}g^{jp}g^{nq}\frac{\partial
g_{pq}}{\partial t}\frac{\partial g_{mn}}{\partial t}
-g^{im}g^{jn}\frac{\partial^{2}g_{mn}}{\partial t^{2}}\right)+\\
&&2g^{im}g^{jn}g^{kp}g^{lq}h_{ik}h_{jl}\frac{\partial
g_{mn}}{\partial t}\frac{\partial g_{pq}}{\partial
t}-8g^{im}g^{jn}g^{kl}h_{jl}\frac{\partial g_{mn}}{\partial
t}\frac{\partial h_{ik}}{\partial t}
+2g^{ij}g^{kl}\frac{\partial h_{ik}}{\partial t}\frac{\partial h_{jl}}{\partial t}\\
&&+2g^{ij}g^{kl}h_{jl}\Big{(}H^{-2}\nabla_{i}\nabla_{k}H-2H^{-3}\nabla_{i}H\nabla_{j}H+H^{-1}h_{ip}g^{pq}h_{qk}\\
&&+g^{pq}h_{ij}\left\langle\vec{\nu}, \frac{\partial^{2}X}{\partial
t\partial x^{p}}\right\rangle \left\langle\vec{\nu},
\frac{\partial^{2}X}{\partial t\partial x^{q}}\right\rangle
+2\frac{\partial\Gamma^{p}_{ij}}{\partial t}\left\langle\vec{\nu}, \frac{\partial^{2}X}{\partial t\partial x^{p}}\right\rangle\Big{)}\\
&=&4g^{im}g^{jp}g^{nq}g^{kl}h_{ik}h_{jl}\frac{\partial
g_{pq}}{\partial t}\frac{\partial g_{mn}}{\partial t}
-2g^{im}g^{jn}g^{kl}h_{ik}h_{jl}\times\\
&&\qquad \left(2H^{-1}h_{mn}+2\left\langle
\frac{\partial^{2}X}{\partial t\partial x^{m}},
\frac{\partial^{2}X}{\partial t\partial x^{n}}\right\rangle\right)\\
&&+2g^{im}g^{jn}g^{kp}g^{lq}h_{ik}h_{jl}\frac{\partial
g_{mn}}{\partial t}\frac{\partial g_{pq}}{\partial t}
-8g^{im}g^{jn}g^{kl}h_{jl}\frac{\partial g_{mn}}{\partial
t}\frac{\partial h_{ik}}{\partial t}
+2g^{ij}g^{kl}\frac{\partial h_{ik}}{\partial t}\frac{\partial h_{jl}}{\partial t}\\
&&+2H^{-2}g^{ij}g^{kl}h_{jl}\nabla_{i}\nabla_{k}H+2H^{-1}tr(A^{3})-4H^{-3}g^{kl}h_{kl}|\nabla H|^{2}\\
&&+2g^{pq}|A|^{2}\left\langle\vec{\nu},
\frac{\partial^{2}X}{\partial t\partial x^{p}}\right\rangle
\left\langle\vec{\nu}, \frac{\partial^{2}X}{\partial t\partial
x^{q}}\right\rangle
+4g^{ij}g^{kl}h_{jl}\frac{\partial\Gamma^{p}_{ik}}{\partial t}
\left\langle\vec{\nu}, \frac{\partial^{2}X}{\partial t\partial
x^{p}}\right\rangle.
\end{eqnarray*}
Noting, by Lemma \ref{lemma4-1},
$$\Delta |A|^{2}=2g^{ik}g^{jl}h_{kl}\nabla_{i}\nabla_{j}H+2|\nabla A|^{2}+2H tr(A^{3})-2|A|^{4},$$
then
\begin{eqnarray*}
\frac{\partial^{2}}{\partial
t^{2}}|A|^{2}&=&H^{-2}\Delta(|A|^{2})-2H^{-2}|\nabla
A|^{2}+2H^{-2}|A|^{4}
-4H^{-2}|\nabla H|^{2}-4H^{-1}tr(A^{3})\\
&&+\frac{2}{\upsilon^2}|A|^{2}g^{pq}u_{pt}u_{qt}
+2g^{ij}g^{kl}\frac{\partial h_{ik}}{\partial t}\frac{\partial
h_{jl}}{\partial t}
-8g^{im}g^{jn}g^{kl}h_{jl}\frac{\partial g_{mn}}{\partial t}\frac{\partial h_{ik}}{\partial t}\\
&&-4g^{im}g^{jn}g^{kl}h_{ik}h_{jl}u_{mt}u_{nt}
+\frac{4}{\upsilon}g^{ij}g^{kl}h_{jl}\frac{\partial\Gamma^{p}_{ik}}{\partial t}u_{pt}\\
&&+2g^{im}h_{ik}h_{jl}\frac{\partial g_{pq}}{\partial t}\frac{\partial g_{mn}}{\partial t}(2g^{jp}g^{nq}g^{kl}+g^{jn}g^{kp}g^{lq})\\
&=&H^{-2}\triangle(|A|^{2})-2H^{-2}|\nabla A|^{2}+2H^{-2}|A|^{4}-4H^{-2}|\nabla H|^{2}-4H^{-1}tr(A^{3})\\
&&+2g^{ij}g^{kl}\frac{\partial h_{ik}}{\partial t}\frac{\partial
h_{jl}}{\partial t}
-8g^{im}g^{jn}g^{kl}h_{jl}\frac{\partial g_{mn}}{\partial t}\frac{\partial h_{ik}}{\partial t}\\
&&+2g^{im}h_{ik}h_{jl}\frac{\partial g_{pq}}{\partial t}\frac{\partial g_{mn}}{\partial t}(2g^{jp}g^{nq}g^{kl}+g^{jn}g^{kp}g^{lq})\\
&&+\frac{2}{\upsilon^2}|A|^{2}g^{pq}e^{2\varphi}(\varphi_{pt}+\varphi_{p}\varphi_{t})(\varphi_{qt}+\varphi_{q}\varphi_{t})
+\frac{4}{\upsilon}g^{ij}g^{kl}h_{jl}\frac{\partial\Gamma^{p}_{ik}}{\partial t}e^{\varphi}(\varphi_{pt}+\varphi_{p}\varphi_{t})\\
&&-4g^{im}g^{jn}g^{kl}h_{ik}h_{jl}e^{2\varphi}(\varphi_{mt}+\varphi_{m}\varphi_{t})(\varphi_{nt}+\varphi_{n}\varphi_{t}),
\end{eqnarray*}
which completes the proof of Lemma \ref{lemma4-6}.
\end{proof}

As we can seen from \emph{complicated} evolution equations in this
section, it is difficult to get gradient estimates and higher-order
estimates for the mean curvature and the second fundamental forms,
which leads to the result that so far we cannot say something about
the convergence of the HIMCF (\ref{IHMCF}) and also the hyperbolic
flows (\ref{IHMCF-1}), (\ref{IHMCF-2}). However, for the lower
dimensional case (i.e., the HIMCF in the plane $\mathbb{R}^2$), we
can get the expanding and convergent conclusions, which will be
shown clearly in the following section.

\section{HIMCF in the plane $\mathbb{R}^2$}
\renewcommand{\thesection}{\arabic{section}}
\renewcommand{\theequation}{\thesection.\arabic{equation}}
\setcounter{equation}{0}

\subsection{The short-time existence}

Consider a family of closed plane curves
$F:\mathbb{S}^{1}\times[0,T)\rightarrow\mathbb{R}^2$ which satisfy
the following IVP
\begin{eqnarray}\label{curve flow}
\left\{\begin{array}{lll}
\frac{\partial^{2}}{\partial
t^{2}}F(u,t)=k^{-1}(u,t)\vec{\nu}(u,t)-\nabla\rho(u,t),~~~~\forall
u\in\mathbb{S}^{1},~t\in[0,T)\\[2mm]
\frac{\partial F}{\partial t}(\cdot,0)=f(u)\vec{\nu}_{0},\\[2mm]
 F(\cdot,0)=F_{0}, &\quad
\end{array}
\right.
\end{eqnarray}
where $k(u,t)$ and $\vec{\nu}$ are the curvature and the unit
outward normal vector of the plane curve $F(u, t)$ respectively,
$f(u)\in C^{\infty}(\mathbb{S}^1)$ is the initial normal velocity,
and $\vec{\nu}_{0}$ is the unit outward normal vector of the smooth
\emph{strictly convex} plane curve $F_{0}(u)$. Besides, $\nabla\rho$
is defined by
\begin{eqnarray*}
\nabla\rho:=\left\langle\frac{\partial^{2}F}{\partial s\partial t},
\frac{\partial F}{\partial t}\right\rangle\vec{T}(u,t),
\end{eqnarray*}
where, by abuse of a notation, $\langle,\rangle$ denotes the
standard Euclidean metric in $\mathbb{R}^{2}$ also, and $\vec{T}$,
$s$ are the unit tangent vector of $F(u,t)$ and the arc-length
parameter respectively.

Now, we would like to show that the HIMCF (\ref{curve flow}) is a
normal flow. However, before that, we need the following definition
which has been mentioned in \cite{klw,m1}.

\begin{definition}
A curve $F:\mathbb{S}^{1}\times [0,T)\rightarrow\mathbb{R}^{2}$
evolves normally if and only if its tangential velocity vanishes.
\end{definition}

\begin{lemma} \label{ADD}
The hyperbolic curvature flow (\ref{curve flow}) is a normal flow.
\end{lemma}

\begin{proof}
By a direct computation, we have
\begin{eqnarray*}
\frac{d}{dt}\left\langle\frac{\partial F}{\partial t},
\frac{\partial F}{\partial
u}\right\rangle&=&\left\langle\frac{\partial^{2} F}{\partial t^{2}},
\frac{\partial F}{\partial u}\right\rangle
+\left\langle\frac{\partial F}{\partial t}, \frac{\partial^{2} F}{\partial t\partial u}\right\rangle\\
&=&\left\langle-\nabla\rho, \frac{\partial F}{\partial
u}\right\rangle +\left\langle\frac{\partial F}{\partial t},
\frac{\partial^{2} F}{\partial t\partial
u}\right\rangle\\
&=&-\left\langle\frac{\partial^{2}F}{\partial s\partial t},
\frac{\partial F}{\partial
t}\right\rangle\cdot\left\langle\frac{\partial F}{\partial
s},\frac{\partial F}{\partial u}\right\rangle
+\left\langle\frac{\partial F}{\partial t}, \frac{\partial^{2}
F}{\partial t\partial u}\right\rangle\\
 &=&0,
\end{eqnarray*}
which, together with the fact that the initial velocity of the IVP
(\ref{curve flow}) is normal, implies the conclusion of Lemma
\ref{ADD}.
\end{proof}

By the IVP (\ref{curve flow}) and Lemma \ref{ADD}, it is easy to get
the following
\begin{eqnarray}\label{norma flow}
\left\{\begin{array}{ll}
\frac{\partial F}{\partial t}(u, t)=\sigma(u, t)\vec{\nu}\\[2mm]
 F(u,0)=F_{0}(u), &\quad
\end{array}
\right.
\end{eqnarray}
where $\sigma(u,t)=f(u)+\int_{0}^{t}k^{-1}(u,\xi)d\xi$. So, we have
\begin{eqnarray*}
\frac{\partial\sigma}{\partial t}=k^{-1}(u,t), \qquad
\sigma\frac{\partial\sigma}{\partial
s}=\left\langle\frac{\partial^{2}F}{\partial s\partial
t},\frac{\partial F}{\partial t}\right\rangle,
\end{eqnarray*}
where $s=s(\cdot, t)$ is the arc-length parameter of curve $F(\cdot,
t):\mathbb{S}^{1}\rightarrow \mathbb{R}^{2}$. By arc-length formula,
we have
\begin{eqnarray*}
\frac{\partial}{\partial s}=\frac{1}{\sqrt{\left(\frac{\partial
x}{\partial u}\right)^{2} +\left(\frac{\partial y}{\partial
u}\right)^{2}}}\frac{\partial}{\partial u}
=\frac{1}{\left|\frac{\partial F}{\partial
u}\right|}\frac{\partial}{\partial u}
:=\frac{1}{\upsilon}\frac{\partial}{\partial u},
 \end{eqnarray*}
  where $(x,y)$ is
the cartesian coordinates, and $\upsilon=\sqrt{\left(\frac{\partial
x}{\partial u}\right)^{2}+\left(\frac{\partial y}{\partial
u}\right)^{2}}=\left|\frac{\partial F}{\partial u}\right|$. For the
orthogonal field $\{\vec{\nu},\vec{T}\}$ of $\mathbb{R}^{2}$, by
Frenet formula, we have
\begin{eqnarray} \label{FORMULA-1}
\frac{\partial\vec{T}}{\partial s}=-k\vec{\nu},\qquad
\frac{\partial\vec{\nu}}{\partial s}=k\vec{T}.
\end{eqnarray}
Denote by $\theta$ the unit inner normal angle for a convex closed
curve $F:\mathbb{S}^{1}\rightarrow\mathbb{R}^{2}$. Then, we have
 \begin{eqnarray*}
\vec{\nu}=(\cos\theta, \sin\theta),\qquad \vec{T}=(-\sin\theta,
\cos\theta).
 \end{eqnarray*}
Together with (\ref{FORMULA-1}),  we have
\begin{eqnarray*}
\frac{\partial\vec{T}}{\partial
s}=\frac{\partial\vec{T}}{\partial\theta}
\frac{\partial\theta}{\partial
s}=-\vec{\nu}\frac{\partial\theta}{\partial s}=-k\vec{\nu},
 \end{eqnarray*}
which implies $\frac{\partial\theta}{\partial s}=k$. Moreover,
\begin{eqnarray} \label{FORMULA-2}
\frac{\partial\vec{\nu}}{\partial t}
=\frac{\partial\vec{\nu}}{\partial\theta}\frac{\partial\theta}{\partial
t} =\frac{\partial\theta}{\partial t}\vec{T}, \qquad
\frac{\partial\vec{T}}{\partial t}
=\frac{\partial\vec{T}}{\partial\theta}\frac{\partial\theta}{\partial
t} =-\frac{\partial\theta}{\partial t}\vec{\nu}.
 \end{eqnarray}

\begin{lemma}\label{evolution of v}
The derivative of $\upsilon$ with respect to $t$ is
$\frac{\partial\upsilon}{\partial t}=k\sigma\upsilon.$
\end{lemma}
\begin{proof}
By a direct computation, we have
\begin{eqnarray*}
\frac{\partial}{\partial t}(\upsilon^{2})&=&\frac{\partial}{\partial
t}\left\langle\frac{\partial F}{\partial u},\frac{\partial
F}{\partial u}\right\rangle =2\left\langle\frac{\partial F}{\partial
u},\frac{\partial^{2} F}{\partial t\partial u}\right\rangle
=2\left\langle\left|\frac{\partial F}{\partial u}\right|\vec{T},\frac{\partial}{\partial u}(\sigma\vec{\nu})\right\rangle\\
&=&2\left\langle\upsilon\vec{T},\sigma\frac{\partial\vec{\nu}}{\partial
u}\right\rangle
=2\left\langle\upsilon\vec{T},\sigma\frac{\partial\vec{\nu}}{\partial s}\frac{\partial s}{\partial u}\right\rangle\\
&=&2\left\langle\upsilon\vec{T},\sigma k\vec{T}\upsilon\right\rangle\\
&=&2\upsilon^{2}k\sigma,
\end{eqnarray*}
which implies the conclusion of Lemma \ref{evolution of v}.
\end{proof}

By Lemma \ref{evolution of v}, we can obtain
\begin{eqnarray*}
\frac{\partial^{2}}{\partial t\partial s}=\frac{\partial}{\partial
t}\left(\frac{1}{\upsilon}\frac{\partial}{\partial u}\right)
=-\frac{1}{\upsilon^{2}}\frac{\partial\upsilon}{\partial
t}\frac{\partial}{\partial u}
+\frac{1}{\upsilon}\frac{\partial}{\partial
u}\frac{\partial}{\partial t} =-k\sigma\frac{\partial}{\partial
s}+\frac{\partial^{2}}{\partial s\partial t}.
\end{eqnarray*}
Therefore, together with (\ref{norma flow}), we have
\begin{eqnarray*}
\frac{\partial\vec{T}}{\partial
t}&=&\frac{\partial^{2}F}{\partial t\partial s}\\
&=&-k\sigma\frac{\partial F}{\partial s}+\frac{\partial^{2}F}{\partial s\partial t}\\
&=&-k\sigma\vec{T}+\frac{\partial}{\partial s}(\sigma\vec{\nu})\\
&=&\frac{\partial\sigma}{\partial s}\vec{\nu},
\end{eqnarray*}
which, combining with (\ref{FORMULA-2}), yields
\begin{eqnarray*}
\frac{\partial\sigma}{\partial s}=-\frac{\partial\theta}{\partial
t}, \qquad \frac{\partial\vec{\nu}}{\partial
t}=-\frac{\partial\sigma}{\partial s}\vec{T}.
\end{eqnarray*}

Assume $F: \mathbb{S}^{1}\times [0,T)\rightarrow\mathbb{R}^{2}$ is a
family of convex curves satisfying the flow (\ref{curve flow}), and
we can use the normal angle to parameterize the evolving curve
$F(\cdot,t)$, which will give the notion of support function used to
get the short-time existence of the flow. Set
\begin{eqnarray*}
 \widetilde{F}(\theta,\tau)=F(u(\theta,\tau),t(\theta,\tau)),
\end{eqnarray*}
where $t(\theta,\tau)=\tau$.
 By the chain rule, we have
 \begin{eqnarray*}
0=\frac{\partial\theta}{\partial\tau}=\frac{\partial\theta}{\partial
u}\frac{\partial u}{\partial\tau}+\frac{\partial\theta}{\partial
t},
\end{eqnarray*}
 and then
 \begin{eqnarray*}
\frac{\partial\theta}{\partial t}=-\frac{\partial\theta}{\partial
u}\frac{\partial u}{\partial\tau} =-\frac{\partial\theta}{\partial
s}\frac{\partial s}{\partial u}\frac{\partial u}{\partial\tau}
=-k\upsilon\frac{\partial u}{\partial\tau}.
 \end{eqnarray*}
 Therefore, a direct calculation yields
\begin{eqnarray*}
\frac{\partial\vec{T}}{\partial\tau}&=&\frac{\partial\vec{T}}{\partial
u}\frac{\partial u}{\partial\tau}
+\frac{\partial\vec{T}}{\partial t}\\
&=&\frac{\partial\vec{T}}{\partial s}\frac{\partial s}{\partial
u}\frac{\partial u}{\partial\tau}
-\frac{\partial\theta}{\partial t}\vec{\nu}\\
&=&-\left(k\upsilon\frac{\partial u}{\partial\tau}+\frac{\partial\theta}{\partial t}\right)\vec{\nu}\\
&=&0,
\end{eqnarray*}
and
\begin{eqnarray*}
\frac{\partial\vec{\nu}}{\partial\tau}&=&\frac{\partial\vec{\nu}}{\partial
u}\frac{\partial u}{\partial\tau}
+\frac{\partial\vec{\nu}}{\partial t}\\
&=&\frac{\partial\vec{\nu}}{\partial s}\frac{\partial s}{\partial
u}\frac{\partial u}{\partial\tau}
+\frac{\partial\theta}{\partial t}\vec{T}\\
&=&\left(k\upsilon\frac{\partial u}{\partial\tau}+\frac{\partial\theta}{\partial t}\right)\vec{T}\\
&=&0,
\end{eqnarray*}
which implies $\vec{v}$ and $\vec{T}$ are independent of the
parameter $\tau$.

  Define the support function of the evolving curve $\widetilde{F}(\theta,\tau)=(x(\theta,\tau),y(\theta,\tau))$ as follows
  \begin{eqnarray*}
S(\theta,\tau)=\langle\widetilde{F}(\theta,\tau),
\vec{\nu}\rangle=x(\theta,\tau)\cos\theta+y(\theta,\tau)\sin\theta.
\end{eqnarray*}
Then we have
\begin{eqnarray*}
S_{\theta}(\theta,\tau)=\langle\widetilde{F}(\theta,\tau),
\vec{T}\rangle =-x(\theta,\tau)\sin\theta+y(\theta,\tau)\cos\theta,
\end{eqnarray*}
and
\begin{eqnarray*}
\left\{
\begin{array}{ll}
x(\theta,\tau)=S\cos\theta-S_{\theta}\sin\theta,\\[2mm]
y(\theta,\tau)=S\sin\theta+S_{\theta}\cos\theta. &\quad
\end{array}
\right.
\end{eqnarray*}
By a direct computation, we have
\begin{eqnarray*}
S_{\theta\theta}+S&=&\langle\widetilde{F}_{\theta}(\theta,\tau),
\vec{T}\rangle +\langle\widetilde{F}(\theta,\tau), -\vec{\nu}\rangle
+\langle\widetilde{F}(\theta,\tau), \vec{\nu}\rangle\\
&=&\langle\widetilde{F}_{\theta}(\theta,\tau), \vec{T}\rangle\\
&=&\left\langle\frac{\partial F}{\partial u}\frac{\partial u}{\partial s}\frac{\partial s}{\partial\theta},\vec{T}\right\rangle\\
&=&\frac{1}{k}.
\end{eqnarray*}
The above expression makes sense, since the evolving curve is
strictly convex.

  On the other hand, since $\vec{\nu}$ and $\vec{T}$ are independent of the parameter $\tau$, together with (\ref{norma flow}) and the definition of the support function $S$,
  we can get
\begin{eqnarray*}
S_{\tau}=\left\langle\frac{\partial\widetilde{F}}{\partial\tau},\vec{\nu}\right\rangle
=\left\langle\frac{\partial F}{\partial u}\frac{\partial
u}{\partial\tau}+\frac{\partial F}{\partial
t},\vec{\nu}\right\rangle =\left\langle\frac{\partial F}{\partial
t},\vec{\nu}\right\rangle=\widetilde{\sigma}(\theta,\tau),
\end{eqnarray*}
and
\begin{eqnarray*}
S_{\tau\tau}&=&\left\langle\frac{\partial^{2}\widetilde{F}}{\partial\tau^{2}},\vec{\nu}\right\rangle\\
&=&\left\langle\frac{\partial F}{\partial
u}\frac{\partial^{2}u}{\partial\tau^{2}}
+\frac{\partial^{2}F}{\partial u^{2}}\left(\frac{\partial
u}{\partial\tau}\right)^{2} +2\frac{\partial^{2}F}{\partial
u\partial t}\frac{\partial u}{\partial\tau}
+\frac{\partial^{2}F}{\partial t^{2}},\vec{\nu}\right\rangle\\
&=&\left\langle\frac{\partial^{2}F}{\partial
u^{2}}\left(\frac{\partial u}{\partial\tau}\right)^{2}
+\frac{\partial^{2}F}{\partial u\partial t}\frac{\partial
u}{\partial\tau},\vec{\nu}\right\rangle
+\left\langle\frac{\partial^{2}F}{\partial u\partial
t}\frac{\partial u}{\partial\tau}
+\frac{\partial^{2}F}{\partial t^{2}},\vec{\nu}\right\rangle\\
&=&\frac{\partial u}{\partial\tau}\left\langle\left(\frac{\partial
F}{\partial u}\right)_{\tau},\vec{\nu}\right\rangle
+\left\langle\frac{\partial^{2}F}{\partial u\partial
t}\frac{\partial u}{\partial\tau}
+\frac{\partial^{2}F}{\partial t^{2}},\vec{\nu}\right\rangle\\
&=&\left\langle\frac{\partial^{2}F}{\partial u\partial
t}\frac{\partial u}{\partial\tau}
+\frac{\partial^{2}F}{\partial t^{2}},\vec{\nu}\right\rangle\\
&=&\left\langle\frac{\partial^{2}F}{\partial u\partial
t}\frac{\partial u}{\partial\tau},\vec{\nu}\right\rangle+k^{-1}.
\end{eqnarray*}
Since $F: \mathbb{S}^{1}\times [0,T)\rightarrow \mathbb{R}^{2}$ is a
normal flow (see Lemma \ref{ADD}), which implies
\begin{eqnarray*}
\left\langle\frac{\partial F}{\partial
t},\vec{T}\right\rangle(u,t)=0,
\end{eqnarray*}
for all $t\in [0,T)$, we have
\begin{eqnarray*}
S_{\tau\theta}=\frac{\partial}{\partial\tau}\left\langle\widetilde{F},\vec{T}\right\rangle
=\left\langle\frac{\partial F}{\partial u}\frac{\partial
u}{\partial\tau} +\frac{\partial F}{\partial
t},\vec{T}\right\rangle=\upsilon\frac{\partial u}{\partial\tau}
\end{eqnarray*}
and
\begin{eqnarray*}
S_{\theta\tau}=\frac{\partial}{\partial\theta}\left\langle\frac{\partial
F}{\partial t},\vec{\nu}\right\rangle
=\left\langle\frac{\partial^{2}F}{\partial u\partial
t}\frac{\partial u}{\partial\theta},\vec{\nu}\right\rangle
=\left\langle\frac{\partial^{2}F}{\partial u\partial
t}\frac{\partial u}{\partial s}\frac{\partial s}{\partial\theta},
\vec{\nu}\right\rangle=\frac{1}{k\upsilon}\left\langle\frac{\partial^{2}F}{\partial
u\partial t},\vec{\nu}\right\rangle
\end{eqnarray*}
by straightforward computation. Hence, $S(\theta,\tau)$ satisfies
\begin{eqnarray*}
S_{\tau\tau}=\left\langle\frac{\partial^{2}F}{\partial u\partial
t}\frac{\partial u}{\partial\tau},\vec{\nu}\right\rangle+k^{-1}
=k\upsilon S_{\theta\tau}\frac{\partial
u}{\partial\tau}+k^{-1}=kS_{\theta\tau}^{2}+k^{-1},
 \end{eqnarray*}
which is equivalent to
 \begin{eqnarray*}
S_{\tau\tau}=\frac{S_{\theta\tau}^{2}}{S_{\theta\theta}+S}+(S_{\theta\theta}+S),
\qquad \forall (\theta,\tau)\in\mathbb{S}^{1}\times[0,T).
\end{eqnarray*}
Together with $(\ref{curve flow})$, we know that
\begin{eqnarray}\label{s flow}
\left\{\begin{array}{lll}
SS_{\tau\tau}+S_{\tau\tau}S_{\theta\theta}-S_{\theta\tau}^{2}-(S_{\theta\theta}+S)^2=0,\\[2mm]
S(\theta,0)=h(\theta),\\[2mm]
S_{\tau}(\theta,0)=\widetilde{f}(\theta),&\quad
\end{array}
\right.
\end{eqnarray}
where $h(\theta)$ and $\widetilde{f}(\theta)$ are the support
function of the initial curve $F_{0}(u(\theta))$ and the initial
velocity of this initial curve respectively.

 Similar to the high-dimensional case mentioned in Section 2, here we would like to get the short-time existence of the IVP (\ref{s flow}) by the linearization method.
 Let
\begin{eqnarray*}
Q(S_{\theta\theta}, S_{\theta\tau},
S):=\frac{S_{\theta\tau}^{2}}{S_{\theta\theta}+S}+(S_{\theta\theta}+S),
 \end{eqnarray*}
 then we have
  \begin{eqnarray} \label{EVEQ}
S_{\tau\tau}=\frac{\partial Q}{\partial
S_{\theta\theta}}S_{\theta\theta}+\frac{\partial Q}{\partial
S_{\theta\tau}}S_{\theta\tau} +\frac{\partial Q}{\partial S}S,
 \end{eqnarray}
  where
\begin{eqnarray*}
\frac{\partial Q}{\partial
S_{\theta\theta}}=1-\frac{S_{\theta\tau}^{2}}{(S_{\theta\theta}+S)^{2}},
\qquad
 \frac{\partial Q}{\partial
S_{\theta\tau}}=\frac{2S_{\theta\tau}}{S_{\theta\theta}+S}.
 \end{eqnarray*}
 Consider the coefficient matrix of terms in (\ref{EVEQ}) involving
 second-order derivatives of $S$, and then we have
\begin{gather*}
\begin{pmatrix}
-1 & \frac{S_{\theta\tau}}{S_{\theta\theta}+S}\\
\frac{S_{\theta\tau}}{S_{\theta\theta}+S} &
1-\frac{S_{\theta\tau}^{2}}{(S_{\theta\theta}+S)^{2}}
\end{pmatrix}
\end{gather*}
which, by a suitable linear transformation, we have
\begin{gather*}
\begin{pmatrix}
-1 & 0\\
0 & 1
\end{pmatrix}
.
\end{gather*}
So (\ref{EVEQ}) is a second-order hyperbolic PDE. By the standard
theory of second-order hyperbolic PDEs, we have the following.

\begin{theorem} \label{main-2}(Local existences and uniqueness) Assume that
$F_{0}$ is a smooth strictly convex closed plane curve. Then there
exist a positive $T_{\rm{max}}>0$ and a family of strictly convex
closed curves $F(u,t)$  satisfying the IVP $(\ref{curve flow})$ on
$\mathbb{S}^{1}\times[0,T_{\rm{max}})$, provided $f(u)$ is a smooth
function on $\mathbb{S}^{1}$.
\end{theorem}

\subsection{Expansion and Convergence}

As in Section 3, we would like to understand further and then try to
get more evolution information about the hyperbolic flow (\ref{curve
flow}) through the following interesting example. For simplicity, we
replace $\tau$ by $t$.

\begin{example}\label{radius}
Let $F(u,t)$ be a family of round circles in $\mathbb{R}^2$ with the
radius $r(t)$ centered at the origin, i.e.,
$$F(u,t)=r(t)(\cos\theta, \sin\theta).$$
Then the support function $S$ and the curvature $k$ are given by
 \begin{eqnarray*}
S(\theta,t)=r(t), \qquad k(\theta,t)=\frac{1}{r(t)},
 \end{eqnarray*}
which implies IVP $(\ref{curve flow})$ becomes
\begin{equation}\label{evolution of r}
\left\{\begin{array}{ll}
r_{tt}=r(t),\\[2mm]
r(0)=r_{0}>0,~~~~r_{t}(0)=r_{1}.&
 \end{array}\right.
\end{equation}
 Solving (\ref{evolution of r}) directly yields
\begin{eqnarray*}
r(t)=\frac{r_{0}+r_{1}}{2}e^{t}+\frac{r_{0}-r_{1}}{2}e^{-t}
\end{eqnarray*}
on $[0,T_{\rm{max}})$ for some $0<T_{\rm{max}}\leqslant\infty$. As
Example \ref{ex1}, we know that
\begin{itemize}
\item if $r_{0}+r_{1}>0$, then $T_{\rm{max}}=\infty$ (i.e., the flow exists for all the time).
Moreover, if furthermore, $r_{0}-r_{1}\leqslant0$, the evolving
curves \emph{expand} exponentially to the infinity, and if
furthermore, $r_{0}-r_{1}>0$, then the evolving curves
\emph{converge} first for a while and then \emph{expand}
exponentially to the infinity;

\item if $r_{0}+r_{1}=0$, then
$r(t)=r_{0}e^{-t}$, which implies $T_{\rm{max}}=\infty$ and the
evolving curves \emph{converge} to a single point as time tends to
infinity;

\item if $r_{0}+r_{1}<0$, then
$T_{\rm{max}}=\frac{1}{2}\ln\left(\frac{r_{1}-r_{0}}{r_{1}+r_{0}}\right)$
and the evolving curves \emph{converge} to a single point as
$t\rightarrow T_{\rm{max}}$.
\end{itemize}
\end{example}

\begin{remark}
\rm{ From the above example, we know that although the initial curve
is so special (i.e., circles), the evolution of the flow (\ref{curve
flow}) is complicated which deeply depends on the initial values of
the flow. It seems like it is very difficult to accurately describe
the evolution of the HIMCF (\ref{curve flow}) as time tends to the
maximal existence time (i.e., as $t\rightarrow T_{\rm{max}}$).
Fortunately, using the containment principle we have derived (see
Proposition \ref{containment} below), we can overcome this
difficulty.}
\end{remark}

 In order to get the containment principle, we need to use the maximum principle for a strip (see Lemma \ref{LEMMA11} below) which has been shown
 in \cite{pw}. However, in order to state the conclusion of Lemma \ref{LEMMA11} clearly, we need to
 introduce some preliminaries first, which has been mentioned in \cite{klw}.
  Consider the general second-order operator $L$
\begin{eqnarray}\label{operator L}
L[\omega]:=a\omega_{\theta\theta}+2b\omega_{\theta
t}+c\omega_{tt}+d\omega_{\theta}+e\omega_{t}
\end{eqnarray}
where $a$, $b$, $c$ are twice continuously differentiable and $d$,
$e$ are continuously differentiable of $\theta$ and $t$. If
$b^{2}-ac>0$ at a point $(\theta,t)$, then the operator $L$ is said
to be hyperbolic at this point. It is hyperbolic in a domain $D$ if
it is hyperbolic at each point of $D$, and uniformly hyperbolic in a
domain $D$ if there exists a constant $\mu$ such that
$b^{2}-ac\geqslant\mu>0$ in $D$.

Assume that $\omega$ and the conormal derivative
 \begin{eqnarray*}
 \frac{\partial\omega}{\partial\nu}\triangleq-b\frac{\partial\omega}{\partial\theta}-c\frac{\partial\omega}{\partial
t}
\end{eqnarray*}
 are given at $t=0$. The
adjoint operator $L^{*}$ associated with $L$ can be defined by
\begin{eqnarray*}
L^{*}[\omega]&\triangleq&(a\omega)_{\theta\theta}+2(b\omega)_{\theta t}+(c\omega)_{tt}-(d\omega)_{\theta}-(e\omega)_{t}\\
&=&a\omega_{\theta\theta}+2b\omega_{\theta t}+c\omega_{tt}+(2a_{\theta}+2b_{t}-d)\omega_{\theta}+(2b_{\theta}+2c_{t}-e)\omega_{t}\\
&&+(a_{\theta\theta}+2b_{\theta t}+c_{tt}-d_{\theta}-e_{t})\omega.
\end{eqnarray*}
Set
\begin{eqnarray*}
K_{+}(\theta,t):&=&\left(\sqrt{b^{2}-ac}\right)_{\theta}+\frac{b}{c}\left(\sqrt{b^{2}-ac}\right)_{\theta}+\frac{1}{c}\left(b_{\theta}+c_{t}-e\right)\sqrt{b^{2}-ac}\\
&&+\left[-\frac{1}{2c}(b^{2}-ac)_{\theta}+a_{\theta}+b_{t}-d-\frac{b}{c}(b_{\theta}+c_{t}-e)\right],
\end{eqnarray*}
and
\begin{eqnarray*}
K_{-}(\theta,t):&=&\left(\sqrt{b^{2}-ac}\right)_{\theta}+\frac{b}{c}\left(\sqrt{b^{2}-ac}\right)_{\theta}+\frac{1}{c}\left(b_{\theta}+c_{t}-e\right)\sqrt{b^{2}-ac}\\
&&-\left[-\frac{1}{2c}(b^{2}-ac)_{\theta}+a_{\theta}+b_{t}-d-\frac{b}{c}(b_{\theta}+c_{t}-e)\right].
\end{eqnarray*}
 As shown in \cite[pp. 502-503]{klw}, we know that for
\begin{eqnarray} \label{DL}
l(\theta, t):=1+\alpha t-\beta t^{2}
\end{eqnarray}
with $\alpha$, $\beta$ sufficiently large such that
\begin{equation}  \label{cofficients}
\left\{\begin{array}{llll}
2\sqrt{b^{2}-ac}(\alpha-2\beta t)+(1+\alpha t-\beta t^{2})K_{+}\geqslant 0\\[2mm]
2\sqrt{b^{2}-ac}(\alpha-2\beta t)+(1+\alpha t-\beta t^{2})K_{-}\geqslant 0\\[2mm]
-2c\beta+(2b_{\theta}+2c_{t}-e)(\alpha-2\beta t)+\\
\qquad \qquad (a_{\theta\theta}+ 2b_{\theta
t}+c_{tt}-d_{\theta}-e_{t} +g)(1 + \alpha t+\beta t^{2})\geqslant0&
 \end{array}\right.
\end{equation}
and $l(\theta,t)>0$ on a sufficiently small strip $0\leqslant
t\leqslant t_{0}$, the hyperbolic operator defined by (\ref{operator
L}) satisfies
\begin{equation*}
\left\{\begin{array}{lll}
2\sqrt{b^{2}-ac}\left[l_{t}-\frac{1}{c}\left(\sqrt{b^{2}-ac}-b\right)l_{\theta}\right]+lK_{+}\geqslant 0\\[2mm]
2\sqrt{b^{2}-ac}\left[l_{t}+\frac{1}{c}\left(\sqrt{b^{2}-ac}-b\right)l_{\theta}\right]+lK_{-}\geqslant 0\\[2mm]
(L^{\ast}+g)[w]\geqslant0&
 \end{array}\right.
\end{equation*}
on the same strip $0\leqslant t\leqslant t_{0}$. It is easy to check
that with $l$ defined as (\ref{DL}), the condition on the conormal
derivative
\begin{eqnarray*}
\frac{\partial\omega}{\partial\nu}+(b_{\theta}+c_{t}-e+c\alpha)\omega\leqslant0,
\end{eqnarray*}
becomes at $t=0$. Besides, if we select a constant $M$ so large that
\begin{eqnarray}\label{M}
M\geqslant -(b_{\theta}+c_{t}-e+c\alpha), \qquad
\rm{on~~\Gamma_{0}},
\end{eqnarray}
then the following maximum principle for the strip adjacent to the
$\theta$-axis can be obtained.

\begin{lemma}\label{LEMMA11}
Suppose that the coefficients of the operator $L$ given by
$(\ref{operator L})$ are bounded and have bounded first and second
derivatives. Let $D$ be an admissible domain. If $t_{0}$ and $M$ are
selected in accordance with $(\ref{cofficients})$ and $(\ref{M})$,
then any function $\omega$ which satisfies
\begin{equation*}
\left\{\begin{array}{ll}
(L+g)[\omega]\geqslant 0 ~~~~~in~~ D,\\[2mm]
\frac{\partial\omega}{\partial\nu}-M\omega\leqslant 0~~~~~on~~ \Gamma_{0},\\[2mm]
\omega\leqslant 0~~~~~~~~~~~~~~~~~~on~~ \Gamma_{0},&
 \end{array}\right.
\end{equation*}
also satisfies $\omega\leqslant 0$ in the part of $D$ which lies in
the strip $0\leqslant t\leqslant t_{0}$. The constants $t_{0}$ and
$M$ depend only on lower bounds for $-c$ and $\sqrt{b^{2}-ac}$ and
on bounds for the coefficients of $L$ and their derivatives.

\end{lemma}

\begin{proposition}\label{containment}(Containment principle)
Let $F_{1}$ and $F_{2}:\mathbb{S}^{1}\times [0,T)\rightarrow
\mathbb{R}^{2}$ be two convex solutions of (\ref{curve flow}).
Suppose that $F_{2}(u,0)$ lies in the domain enclosed by
$F_{1}(u,0)$, and $f_{2}(u)\leqslant f_{1}(u)$. Then $F_{2}(u,t)$ is
contained in the domain enclosed by $F_{1}(u,t)$ for all $t\in
[0,T)$.
\end{proposition}

\begin{proof}
Let $S_{1}(\theta,t)$ and $S_{2}(\theta,t)$ be the support functions
of $F_{1}(u,t)$ and $F_{2}(u,t)$, respectively. Then
$S_{1}(\theta,t)$ and $S_{2}(\theta,t)$ satisfy the same equation
(\ref{s flow}) with $S_{2}(\theta,0)\leqslant S_{1}(\theta,0)$ and
$S_{2t}(\theta,0)\leqslant S_{1t}(\theta,0)$.

 Let
\begin{eqnarray*}
 \omega(\theta,t):=S_{2}(\theta,t)-S_{1}(\theta,t).
 \end{eqnarray*}
 Then we
have
\begin{eqnarray*}
\omega_{tt}&=&S_{2tt}-S_{1tt}=\frac{S_{2\theta
t}^{2}+k_{2}^{-2}}{S_{2}+S_{2\theta\theta}}
-\frac{S_{1\theta t}^{2}+k_{1}^{-2}}{S_{1}+S_{1\theta\theta}}\\
&=&k_{1}k_{2}\left(\frac{1}{k_{1}k_{2}}-S_{1\theta t}S_{2\theta
t}\right)\omega_{\theta\theta}+(k_{1}S_{1\theta t}+k_{2}S_{2\theta
t})\omega_{\theta t}
+k_{1}k_{2}\left(\frac{1}{k_{1}k_{2}}-S_{1\theta t}S_{2\theta
t}\right)\omega,
\end{eqnarray*}
which implies that $\omega$ satisfies the following system
\begin{equation}\label{W flow}
\left\{\begin{array}{ll}
\omega_{tt}=k_{1}k_{2}\left(\frac{1}{k_{1}k_{2}}-S_{1\theta
t}S_{2\theta t}\right)\omega_{\theta\theta}+(k_{1}S_{1\theta
t}+k_{2}S_{2\theta t})\omega_{\theta t}
+k_{1}k_{2}\left(\frac{1}{k_{1}k_{2}}-S_{1\theta t}S_{2\theta t}\right)\omega,\\[2mm]
\omega_{t}(\theta,0)=f_{2}(\theta)-f_{1}(\theta)=\omega_{1}(\theta),\\[2mm]
\omega(\theta,0)=h_{2}(\theta)-h_{1}(\theta)=\omega_{0}(\theta).&
 \end{array}\right.
\end{equation}
Define the operator $L$ by
 \begin{eqnarray*}
 L[\omega]:=k_{1}k_{2}\left(\frac{1}{k_{1}k_{2}}-S_{1\theta
t}S_{2\theta t}\right)\omega_{\theta\theta}+(k_{1}S_{1\theta
t}+k_{2}S_{2\theta t})\omega_{\theta t}-\omega_{tt},
 \end{eqnarray*}
  and then we know that
\begin{eqnarray*}
 a=k_{1}k_{2}\left(\frac{1}{k_{1}k_{2}}-S_{1\theta t}S_{2\theta
t}\right),~~~b=\frac{1}{2}(k_{1}S_{1\theta t}+k_{2}S_{2\theta
t}),~~~c=-1
 \end{eqnarray*}
are twice continuously differentiable functions of $\theta$ and $t$.
By a direct computation, we have
\begin{eqnarray*}
b^{2}-ac&=&\frac{1}{4}(k_{1}S_{1\theta t}+k_{2}S_{2\theta t})^{2}-k_{1}k_{2}\left(\frac{1}{k_{1}k_{2}}-S_{1\theta t}S_{2\theta t}\right)\cdot (-1)\\
&=&\frac{1}{4}(k_{1}S_{1\theta t}-k_{2}S_{2\theta t})^{2}+1>0.
\end{eqnarray*}
Hence, the operator $L$ is uniformly hyperbolic in
$\mathbb{S}^{1}\times [0,T)$. By Lemma \ref{LEMMA11}, we deduce that
$S_{2}(\theta,t)\leqslant S_{1}(\theta,t)$ for all $t\in [0,T)$,
which completes the proof.
\end{proof}

\begin{proposition}\label{convex}(Preserving convexity)
Let $k_{0}(\theta)$ be the curvature function of $F_{0}$ and
$$\delta=\min\limits_{\theta\in[0,2\pi]}\{k_{0}(\theta)\}>0.$$
Then for a $C^4$-solution $S$ of (\ref{s flow}), we have
 \begin{eqnarray*}
  k(\theta, t)\geqslant\delta,
  \end{eqnarray*}
for all $t\in [0,T_{\rm{max}})$, where $[0,T_{\rm{max}})$ is the
maximal time interval for solution $F(\cdot,t)$ of $(\ref{curve
flow})$.
\end{proposition}

\begin{proof}
Since the initial curve is strictly convex, by Theorem \ref{main-2}
we know that the solution of (\ref{s flow}) remains strictly convex
on some short time interval $[0,T)$ with some $T\leqslant
T_{\rm{max}}$ and its support function satisfies
 \begin{eqnarray*}
 S_{tt}=kS_{\theta t}^{2}+k^{-1}
 \end{eqnarray*}
for all $(\theta,t)\in \mathbb{S}^{1}\times[0,T)$. Taking derivative
with respect to $t$, we have
 \begin{eqnarray*}
k_{t}=\left(\frac{1}{S+S_{\theta\theta}}\right)_{t}
=-\frac{1}{(S+S_{\theta\theta})^{2}}(S_{t}+S_{\theta\theta t})
=-k^{2}(S_{t}+S_{\theta\theta t}).
 \end{eqnarray*}
Together with the fact $S_{t}=\widetilde{\sigma}$, it is easy to
know
$k_{t}=-k^{2}(\widetilde{\sigma}+\widetilde{\sigma}_{\theta\theta})$.
Therefore, we can obtain the followings
$$S_{t}+S_{\theta\theta t}=-(S+S_{\theta\theta})^{2}k_{t}=-\frac{1}{k^{2}}k_{t},$$
$$S_{\theta t}+S_{\theta\theta\theta t}=\left(-\frac{1}{k^{2}}k_{t}\right)_{\theta}=\frac{2}{k^{3}}k_{t}k_{\theta}-\frac{1}{k^{2}}k_{\theta t},$$
and
\begin{eqnarray*}
k_{tt}&=&\left(-\frac{1}{(S+S_{\theta\theta})^{2}}(S_{t}+S_{\theta\theta t})\right)_{t}\\
&=&\frac{2}{(S+S_{\theta\theta})^{3}}(S_{t}+S_{\theta\theta t})^{2}-\frac{1}{(S+S_{\theta\theta})^{2}}(S_{tt}+S_{\theta\theta tt})\\
&=&2k^{3}(-\frac{1}{k^{2}}k_{t})^{2}-k^{2}[(S_{tt})_{\theta\theta}+S_{tt}]\\
&=&\frac{2}{k}k_{t}^{2}-k^{2}[(kS_{\theta t}^{2}-k+k+k^{-1})_{\theta\theta}+(kS_{\theta t}^{2}-k+k+k^{-1})]\\
&=&\frac{2}{k}k_{t}^{2}-k^{2}[((S_{\theta t}^{2}-1)k)_{\theta\theta}+(S_{\theta t}^{2}-1)k+(k+k^{-1})_{\theta\theta}+(k+k^{-1})]\\
&=&\frac{2}{k}k_{t}^{2}-k^{2}[((S_{\theta
t}^{2}-1)_{\theta}k+(S_{\theta t}^{2}-1)k_{\theta})_{\theta}
+(S_{\theta t}^{2}-1)k+(k+k^{-1})_{\theta\theta}+(k+k^{-1})]\\
&=&\frac{2}{k}k_{t}^{2}-k^{2}[(S_{\theta
t}^{2}-1)_{\theta\theta}k+2(S_{\theta
t}^{2}-1)_{\theta}k_{\theta}+(S_{\theta t}^{2}-1)k_{\theta\theta}
+(S_{\theta t}^{2}-1)k]\\
&&-k^{2}[(k+k^{-1})_{\theta\theta}+(k+k^{-1})]\\
&=&\frac{2}{k}k_{t}^{2}-k^{2}(S_{\theta
t}^{2}-1)(k+k_{\theta\theta})
-k^{2}[(2S_{\theta t}S_{\theta\theta t})_{\theta}k+4S_{\theta t}S_{\theta\theta t}k_{\theta}]\\
&&-k^{2}[k_{\theta\theta}-\frac{1}{k^{2}}k_{\theta\theta}+\frac{2}{k^{3}}k_{\theta}^{2}+(k+k^{-1})]\\
&=&\frac{2}{k}k_{t}^{2}-k^{2}(S_{\theta
t}^{2}-1)(k+k_{\theta\theta})
-k^{2}[2(S_{\theta\theta t}^{2}+S_{\theta t}S_{\theta\theta\theta t})k+4k_{\theta}S_{\theta t}(S_{\theta\theta}+S-S)_{t}]\\
&&-k^{2}[(1-\frac{1}{k^{2}})k_{\theta\theta}+\frac{2}{k^{3}}k_{\theta}^{2}+(k+k^{-1})]\\
&=&\frac{2}{k}k_{t}^{2}-k^{2}(S_{\theta
t}^{2}-1)(k+k_{\theta\theta})
-k^{2}[2((S_{\theta\theta t}+S_{t})^{2}-2S_{\theta\theta t}S_{t}-S_{t}^{2}\\
&&+S_{\theta t}(S_{\theta\theta}+S)_{\theta t}-S_{\theta t}^{2})k
+4k_{\theta}S_{\theta t}(\frac{1}{k}-S)_{t}]\\
&&-k^{2}[(1-\frac{1}{k^{2}})k_{\theta\theta}+\frac{2}{k^{3}}k_{\theta}^{2}+(k+k^{-1})]\\
&=&\frac{2}{k}k_{t}^{2}-k^{2}(S_{\theta
t}^{2}-1)(k+k_{\theta\theta})
-k^{2}[2((S_{\theta\theta t}+S_{t})^{2}-2(S_{\theta\theta t}+S_{t})S_{t}+S_{t}^{2}\\
&&+S_{\theta t}(\frac{1}{k})_{\theta t}-S_{\theta t}^{2})k
-4k_{\theta}S_{\theta t}\frac{1}{k^{2}}k_{t}-4k_{\theta}S_{\theta t}S_{t}]\\
&&-k^{2}[(1-\frac{1}{k^{2}})k_{\theta\theta}+\frac{2}{k^{3}}k_{\theta}^{2}+(k+k^{-1})]\\
&=&\frac{2}{k}k_{t}^{2}-k^{2}(S_{\theta
t}^{2}-1)(k+k_{\theta\theta})
-2k^{3}[(-\frac{1}{k^{2}}k_{t})^{2}-2(-\frac{1}{k^{2}}k_{t})S_{t}+S_{t}^{2}-S_{\theta t}^{2}\\
&&+S_{\theta
t}(\frac{2}{k^{3}}k_{t}k_{\theta}-\frac{1}{k^{2}}k_{\theta t})]
+4k^{2}(k_{\theta}S_{\theta t}\frac{1}{k^{2}}k_{t}+k_{\theta}S_{\theta t}S_{t})\\
&&-k^{2}[(1-\frac{1}{k^{2}})k_{\theta\theta}+\frac{2}{k^{3}}k_{\theta}^{2}+(k+k^{-1})]\\
&=&k^{2}(1-S_{\theta
t}^{2})(k+k_{\theta\theta})-4kS_{t}k_{t}-2k^{3}S_{t}^{2}+2k^{3}S_{\theta
t}^{2}
+2kS_{\theta t}k_{\theta t}+4k^{2}S_{\theta t}S_{t}k_{\theta}\\
&&-k^{2}[(1-\frac{1}{k^{2}})k_{\theta\theta}+\frac{2}{k^{3}}k_{\theta}^{2}+(k+k^{-1})]\\
&=&k^{2}(\frac{1}{k^{2}}-S_{\theta
t}^{2})k_{\theta\theta}+2kS_{\theta t}k_{\theta t}+4k^{2}S_{\theta
t}S_{t}k_{\theta}-\frac{2}{k}k_{\theta}^{2}-4kS_{t}k_{t}+k^{3}(S_{\theta
t}^{2}-2S_{t}^{2}-k^{-2}).
\end{eqnarray*}
So, the curvature $k$ satisfies the equation
\begin{eqnarray*}
k_{tt}=k^{2}\left(\frac{1}{k^{2}}-S_{\theta
t}^{2}\right)k_{\theta\theta}+2kS_{\theta t}k_{\theta
t}+4k^{2}S_{\theta
t}S_{t}k_{\theta}-\frac{2}{k^{3}}k_{\theta}^{2}-4kS_{t}k_{t}+k^{3}(S_{\theta
t}^{2}-2S_{t}^{2}-k^{-2}).
\end{eqnarray*}
Define the operator $L$ as
 \begin{eqnarray*}
 L[k]:=k^{2}\left(\frac{1}{k^{2}}-S_{\theta
t}^{2}\right)k_{\theta\theta}+2kS_{\theta t}k_{\theta t}-k_{tt}
+4k^{2}S_{\theta
t}S_{t}k_{\theta}-\frac{2}{k^{3}}k_{\theta}^{2}-4kS_{t}k_{t}.
 \end{eqnarray*}
 We
know that
$$a=k^{2}\left(\frac{1}{k^{2}}-S_{\theta t}^{2}\right),~~~b=kS_{\theta t},~~~c=-1$$
are twice continuously differentiable functions of $\theta$ and $t$.
So we have
$$b^{2}-ac=(kS_{\theta t})^{2}-k^{2}\left(\frac{1}{k^{2}}-S_{\theta t}^{2}\right)\cdot (-1)=1>0,$$
which implies that the operator $L$ is hyperbolic in
$\mathbb{S}^{1}\times [0,T)$.

Determining a function $k(\theta,t)$ which satisfies the following
system
\begin{equation*}
\left\{\begin{array}{ll}
(L+\widetilde{h})[k]=0~~~~~~~~~~~\qquad\qquad \qquad \rm{in}~~ \mathbb{S}^{1}\times [0,T),\\[2mm]
k(\theta,0)=k_{0}(\theta)~~~~~~~~\qquad\qquad \qquad \rm{on}~~ \Gamma_{0},\\[2mm]
0\leqslant \frac{\partial
k}{\partial\nu}:=-bk_{\theta}-ck_{t}:=\beta(\theta)~~~~~~~~\rm{on}~~
\Gamma_{0},&
 \end{array}\right.
\end{equation*}
where the operator $\widetilde{h}$ is defined as
$\widetilde{h}[k]:=k^{3}(S_{\theta t}^{2}-2S_{t}^{2}-k^{-2})$. It is
easy to check that the function $\widetilde{k}(\theta,t)=\min
\limits_{\theta\in[0,2\pi]}\{k_{0}(\theta)\}=\delta$ satisfies
\begin{equation*}
\left\{\begin{array}{ll}
(L+\widetilde{h})[\widetilde{k}]=0~~~~~~~~~~~~~~~~~~~~~~~~~~~~~~~\rm{in}~~ \mathbb{S}^{1}\times [0,T),\\[2mm]
\widetilde{k}(\theta,0)\leqslant k_{0}(\theta)~~~~~~~~~~~~~~~~~~~~~~~~~~~~\rm{on}~~ \Gamma_{0},\\[2mm]
\frac{\partial
\widetilde{k}}{\partial\vec{\nu}}-M\widetilde{k}\leqslant\beta(\theta)-Mk_{0}(\theta)~~~~~~~~~\rm{on}~~
\Gamma_{0},&
 \end{array}\right.
\end{equation*}
where $\Gamma_{0}$ is the initial domain, and $M$ is the constant
determined by (\ref{M}). Applying Lemma \ref{LEMMA11} to
$\widetilde{k}-k$ yields
$$\widetilde{k}\leqslant k(\theta,t)~~~~~~{\rm{in}}~~ \mathbb{S}^{1}\times [0,t_{0}).$$
with $t_{0}\leqslant T$. Hence, we know that the solution
$F(\cdot,t)$ remains convex on $[0,T_{\rm{max}})$ and its curvature
function $k(\theta,t)$ has a uniformly positive lower bound
$\delta=\min \limits_{\mathbb{S}^1}\{k_{0}(\theta)\}$ on
$\mathbb{S}^{1}\times [0,T_{\max})$, which completes the proof.
\end{proof}

We need the following evolution equations of the arc-length of
evolving curves.
\begin{lemma}\label{length}
The arc-length $\mathcal{L}(t)$ of the closed curve $F(u,t)$
satisfies
 \begin{eqnarray*}
 \frac{d\mathcal{L}(t)}{dt}=\int_{0}^{2\pi}\widetilde{\sigma}(\theta,t)d\theta,
 \end{eqnarray*}
and
 \begin{eqnarray*}
\frac{d^{2}\mathcal{L}(t)}{dt^{2}}=\int_{0}^{2\pi}\left[k\left(\frac{\partial\widetilde{\sigma}}{\partial\theta}\right)^{2}+k^{-1}\right]d\theta.
 \end{eqnarray*}
\end{lemma}
\begin{proof}
Since
 \begin{eqnarray*}
 \mathcal{L}(t)=\int_{0}^{2\pi}\upsilon(\theta,t)d\theta,
  \end{eqnarray*}
  and $\frac{\partial\upsilon}{\partial t}=k\upsilon\widetilde{\sigma}$,
by a direct calculation, we have
 \begin{eqnarray*}
\frac{d\mathcal{L}(t)}{dt}=\int_{0}^{2\pi}\frac{\partial\upsilon}{\partial
t}d\theta =\int_{0}^{2\pi}k\upsilon\widetilde{\sigma}d\theta=
\int_{0}^{2\pi}\widetilde{\sigma}(\theta,t)d\theta,
 \end{eqnarray*}
 and
\begin{eqnarray*}
\frac{d^{2}\mathcal{L}(t)}{dt^{2}}&=&\int_{0}^{2\pi}\frac{\partial}{\partial
t}\widetilde{\sigma}(\theta,t)d\theta=\int_{0}^{2\pi}S_{tt}d\theta\\
&=&\int_{0}^{2\pi}\left(kS_{\theta t}^{2}+k^{-1}\right)d\theta
=\int_{0}^{2\pi}\left[k\left(\frac{\partial}{\partial\theta}S_{t}\right)^{2}+k^{-1}\right]d\theta\\
&=&\int_{0}^{2\pi}\left[k\left(\frac{\partial\widetilde{\sigma}}{\partial\theta}\right)^{2}+k^{-1}\right]d\theta,
\end{eqnarray*}
 which completes the proof of Lemma \ref{length}.
\end{proof}

From Example \ref{radius}, we know that the behavior of evolving
plane curves of HIMCF (\ref{curve flow}) is complicated. However,
using Propositions \ref{containment} and \ref{convex}, Lemma
\ref{length}, we can get the following conclusion about the
asymptotic behavior of the hyperbolic flow (\ref{curve flow}).

\begin{theorem} \label{main-3}

Suppose that $F_{0}$ is a smooth strictly convex closed plane curve
with the curvature function $k_{0}(\theta)$ whose minimum and
maximum are given by
$\delta=\min\limits_{\mathbb{S}^1}\{k_{0}(\theta)\}>0$ and
$\zeta:=\max\limits_{\mathbb{S}^1}\{k_{0}(\theta)\}$ respectively.
 Then there exists a
family of strictly convex closed plane curves $F(\cdot,t)$
satisfying the IVP $(\ref{curve flow})$ on the time interval
$[0,T_{\rm{max}})$ with $0<T_{\rm{max}}\leqslant\infty$.
Moreover, we have \\
$(I)$ if $\zeta^{-1}+\min\limits_{u\in\mathbb{S}^{1}}f(u)>0$, then
$T_{\rm{max}}=\infty$, i.e., the flow exists for
all the time;\\
 $(II)$ if $\delta^{-1}+\max\limits_{u\in\mathbb{S}^{1}}f(u)<0$, then $T_{\rm{max}}<\infty$.
 Moreover, if furthermore $\delta^{-1}T_{\rm{max}}+\max\limits_{u\in\mathbb{S}^{1}}f(u)<0$, then as $t\rightarrow T_{\rm{max}}$, one of the following must be true:
\begin{itemize}
\item  the solution $F(\cdot,t)$ converges to a point as
$t\rightarrow T_{\rm{max}}$, i.e., the curvature of the limit curve
becomes unbounded as $t\rightarrow T_{\rm{max}}$;

 \item the curvature $k$ of the evolving
curve is discontinuous as $t\rightarrow T_{\rm{max}}$, so the
solution $F(\cdot,t)$ converges to a piecewise smooth curve.
\end{itemize}
\end{theorem}

\begin{remark}
\rm{ In Case $(II)$ of Theorem \ref{main-3} above, the condition
$\delta^{-1}T_{\rm{max}}+\max\limits_{u\in\mathbb{S}^{1}}f(u)<0$ is
not easy to check, since for a general strictly convex closed plane
curve evolving under the hyperbolic flow $(\ref{curve flow})$, it is
difficult to get the accurate value of the maximal time
$T_{\rm{max}}$. However, as shown in the proof below, by Example
\ref{radius} and Proposition \ref{containment} (Containment
principle), we have $T_{\rm{max}}\leqslant
T^{\ast}=\frac{1}{2}\ln\left(\frac{-1+\delta\max\limits_{u\in\mathbb{S}^{1}}f(u)}{1+\delta\max\limits_{u\in\mathbb{S}^{1}}f(u)}\right)$.
So, for the purpose of easily checking, one can use a weaker
condition
$\delta^{-1}T^{\ast}+\max\limits_{u\in\mathbb{S}^{1}}f(u)<0$ to
replace the assumption
$\delta^{-1}T_{\rm{max}}+\max\limits_{u\in\mathbb{S}^{1}}f(u)<0$.
However, here we prefer to use the latter one, since it is sharper
than the previous one. }
\end{remark}

\begin{proof}
Let $[0,T_{\rm{max}})$ be the maximal time interval of the IVP
$(\ref{curve flow})$ with $F_{0}$ and $f$ as the initial curve and
initial velocity of the initial curve, respectively.

By Proposition \ref{convex}, we know that the solution $F(\cdot,t)$
remains strictly convex on $[0,T_{\rm{max}})$ and the curvature of
$F(\cdot,t)$ has a uniformly positive lower bound $\delta>0$ on
$\mathbb{S}^{1}\times [0,T_{\rm{max}})$.

Case $(I)$: When
$\zeta^{-1}+\min\limits_{u\in\mathbb{S}^{1}}f(u)>0$.

Since
$\zeta=\max\limits_{\mathbb{S}^1}\{k_{0}(\theta)\}\geqslant\delta>0$,
the initial curve $F_{0}$ can enclose a circle $\mathcal {C}_0$ with
radius $\zeta^{-1}$. Let the normal initial velocity of $\mathcal
{C}_0$ be equal to $\min\limits_{u\in\mathbb{S}^{1}}f(u)$. Evolving
$\mathcal {C}_0$ by the hyperbolic flow $(\ref{curve flow})$ to get
a solution $\mathcal{C}(\cdot,t)$. By Example \ref{radius}, we know
that if $\zeta^{-1}+\min\limits_{u\in\mathbb{S}^{1}}f(u)>0$, the
evolving circle $\mathcal{C}(\cdot,t)$ exists for all the time, and
its radius tends to infinity as $t\rightarrow\infty$. By Proposition
\ref{containment}, we can get that $\mathcal{C}(\cdot,t)$ always
lies in the domain $\mathcal{D}$ enclosed by the closed curve
$F(\cdot,t)$ for all $t\geqslant0$, and moreover, $\mathcal{D}$
tends to be the whole plane as $t\rightarrow\infty$. So, in this
case, the IVP $(\ref{curve flow})$ has the long-time existence,
i.e., $T_{\rm{max}}=\infty$.

Case $(II)$: When
$\delta^{-1}+\max\limits_{u\in\mathbb{S}^{1}}f(u)<0$.

Since $\delta=\min\limits_{\mathbb{S}^1}\{k_{0}(\theta)\}>0$, the
initial curve $F_{0}$ can be enclosed by a circle $\mathcal {C}_1$
with radius $\delta^{-1}$. Let the normal initial velocity of
$\mathcal {C}_1$ be equal to $\max\limits_{u\in\mathbb{S}^{1}}f(u)$.
Evolving $\mathcal {C}_1$ by the hyperbolic flow $(\ref{curve
flow})$ to get a solution $\widetilde{\mathcal{C}}(\cdot,t)$. By
Example \ref{radius}, we know that if
$\delta^{-1}+\max\limits_{u\in\mathbb{S}^{1}}f(u)<0$, the solution
exist at a finite time interval $[0,T^{\ast})$ and the evolving
circle $\widetilde{\mathcal{C}}(\cdot,t)$ converges to a single
point as $t\rightarrow T^{\ast}$. By Proposition \ref{containment},
we know that the evolving curve $F(\cdot,t)$ always lies in the
domain $\widetilde{\mathcal{D}}$ (i.e., a disk) enclosed by
$\widetilde{\mathcal{C}}(\cdot,t)$ for all $t\in[0,T^{\ast})$.
Hence, we can get that $F(\cdot,t)$ must become singular at some
time $T_{\rm{max}}\leqslant T^{\ast}<\infty$.

Now, we need the following conclusion in convex geometry (see, e.g.,
\cite{RSC}).

\textbf{Blaschke Selection Theorem} \emph{Let $K_{j}$ be a sequence
of convex sets which are contained in a bounded set. Then there
exists a subsequence $K_{jk}$ and a convex set $K$ such that
$K_{jk}$ converges to $K$ in the Hausdorff metric.}

In Case $(II)$, since $\widetilde{\mathcal{C}}(\cdot,t)$ shrinks as
$t$ increases and the evolving curve $F(\cdot,t)$ is contained by
the circle $\widetilde{\mathcal{C}}(\cdot,t)$ for each
$t\in[0,T_{\rm{max}})$, this strictly convex closed plane curve
$F(\cdot,t)$ is contained in the circle $\mathcal {C}_1$ for all
$t\in[0,T_{\rm{max}})$. By Blaschke Selection Theorem, we know that
in the sense of the \emph{Hausdorff metric},  $F(\cdot,t)$ converges
to a weakly convex curve $F(\cdot,T_{\rm{max}})$ which might be
degenerated and non-smooth.

We \textbf{claim} that \emph{$F(\cdot,t)$ converges to either a
single point or a limit curve which has the discontinuous curvature
under the further assumption
$\delta^{-1}T_{\rm{max}}+\max\limits_{u\in\mathbb{S}^{1}}f(u)<0$}.

By Proposition \ref{convex} and Lemma \ref{length}, we have
 \begin{eqnarray*}
\frac{d^{2}\mathcal{L}(t)}{dt^{2}}=\int_{0}^{2\pi}\left[k\left(\frac{\partial\widetilde{\sigma}}{\partial\theta}\right)^{2}+k^{-1}\right]d\theta>0
\qquad \qquad {\rm{for~~all}}~t\in[0,T_{\rm{max}}).
 \end{eqnarray*}
Besides, by Proposition \ref{convex}, we have
\begin{eqnarray*}
\widetilde{\sigma}(\theta,t)=\sigma(u,t)=f(u)+\int_{0}^{t}k^{-1}(u,\xi)d\xi\leqslant\delta^{-1}t+\max\limits_{u\in\mathbb{S}^{1}}f(u)
\leqslant\delta^{-1}T_{\rm{max}}+\max\limits_{u\in\mathbb{S}^{1}}f(u)<0
\end{eqnarray*}
for all $t\in[0,T_{\rm{max}})$, which implies
 \begin{eqnarray*}
 \frac{d\mathcal{L}(t)}{dt}=\int_{0}^{2\pi}\widetilde{\sigma}(\theta,t)d\theta<0 \qquad \qquad {\rm{for~~all}}~t\in[0,T_{\rm{max}}).
 \end{eqnarray*}
So, for all $t\in[0,T_{\rm{max}})$, we have
\begin{eqnarray*}
\frac{d\mathcal{L}(t)}{dt}<0, \qquad \qquad
\frac{d^{2}\mathcal{L}(t)}{dt^{2}}>0,
\end{eqnarray*}
which implies that there exists a finite time $T_{0}$ such that
$\mathcal{L}(T_{0})=0$. There will be the following two situations:

\begin{itemize}
\item $T_{0}\leqslant T_{\rm{max}}$. On one hand, by Theorem
\ref{main-2}, there exists a unique classical solution $F(\cdot,t)$
to the IVP $(\ref{curve flow})$ on $[0,T_{0})$. On the other hand,
since $\mathcal{L}(t)$ is decreasing on $[0,T_{0})$ and
$\mathcal{L}(T_{0})=0$, we have $\mathcal{L}(T_{0})\rightarrow0$ as
$t\rightarrow T_{0}$. This implies the curvature $k$ tends to
infinity as $t\rightarrow T_{0}$, and the solution will blow up at
$T_{0}$. Therefore, by the definition of $T_{\rm{max}}$, we have
$T_{0}=T_{\rm{max}}$. So, $F(\cdot,t)$ converges to a point as
$t\rightarrow T_{\rm{max}}$.

\item $T_{0}> T_{\rm{max}}$. In this situation,
$\mathcal{L}(T_{\rm{max}})>0$, which implies that
$F(\cdot,T_{\rm{max}})$ must be non-smooth. Then there will be three
possibilities:

(1) $\|F(u,T_{\rm{max}})\|=\sup|F(u,T_{\rm{max}})|=\infty$. However,
$F(\cdot,t)$ is always contained in the circle $\mathcal {C}_1$,
which implies that $\|F(u,T_{\rm{max}})\|$ must be bounded. This is
a contradiction. So, (1) is impossible.

(2) $\|F_{u}(u,T_{\rm{max}})\|=\infty$. However, the length of the
limit curve $\mathcal{L}(T_{\rm{max}})$ satisfies
\begin{eqnarray*}
\mathcal{L}(T_{max})&=&\lim\limits_{t\rightarrow T_{\rm{max}}}\int_{F(u,t)}ds\\
&=&\lim\limits_{t\rightarrow T_{\rm{max}}}\int_{F(u,t)}|F_{u}(u,t)|du\\
&=&\int_{F(u,t)}\lim\limits_{t\rightarrow T_{\rm{max}}}|F_{u}(u,t)|du\\
&=&\infty
\end{eqnarray*}
which contradicts with $\mathcal{L}(T_{\rm{max}})<\mathcal{L}_{0}$
with $\mathcal{L}_{0}$ the length of the initial curve $F_{0}$. So,
(2) is also impossible.

(3) The curvature function $k$ is discontinuous. We cannot exclude
this possibility. This phenomena will be occurred if the above
shocks are not possible.
\end{itemize}
Our \textbf{claim} before is true. The proof of Theorem \ref{main-3}
is finished.
\end{proof}

\section*{Acknowledgments}
\renewcommand{\thesection}{\arabic{section}}
\renewcommand{\theequation}{\thesection.\arabic{equation}}
\setcounter{equation}{0} \setcounter{maintheorem}{0}

This work was partially supported by the NSF of China (Grant Nos.
11401131 and 11641001) and Key Laboratory of Applied Mathematics of
Hubei Province (Hubei University). The main result of this paper has
been announced by the corresponding author, Prof. Jing Mao, in
``\emph{Conference of Differential Geometry and its Applications
2017}" hold at Faculty of Mathematics and Statistics, Hubei
University on 28$^{th}$-29$^{th}$ May, 2017. The authors would like
to thank master students Chi Xu, Dan-Dan Hu of Faculty of
Mathematics and Statistics, Hubei University for many useful
discussions in the Geometric Seminar organized by Prof. Jing Mao.


\begin{thebibliography}{99}

\bibitem{hb} H. Bray, \emph{Proof of the Riemannian Penrose inequality using the positive mass
theorem}, J. Differential Geom. {\bf 59} (2001) 177--267.

\bibitem{bhw}  S. Brendle, P.-K. Hung and M.-T. Wang, \emph{A Minkowski inequality for hypersurfaces in the anti-de Sitter-Schwarzschild manifold},
Commun. Pure Appl. Math. {\bf 69} (2016) 124--144.

\bibitem{cns} L. Caffarelli, L. Nirenberg and J. Spruck, \emph{The Dirichlet problem for
nonlinear second order elliptic equations, III; Functions of the
eigenvalue of the Hessian}, Acta Math. {\bf 155} (1985) 261--301.

\bibitem{cf} F. Cao, \emph{Geometric Curve Evolution and Image
Processing}, Lecture Notes in Mathematics, Vol. 1805, Springer,
2003.

\bibitem{cm} L. Chen and J. Mao, \emph{Non-parametric inverse curvature flows in the AdS-Schwarzschild
manifold}, The Journal of Geometric Analysis,
DOI:10.1007/s12220-017-9848-6.

\bibitem{cmz} L. Chen, J. Mao and H.-Y. Zhou, \emph{Inverse curvature flows in warped product
manifolds}, preprint.

\bibitem{cmxc} L. Chen, J. Mao, N. Xiang and C. Xu, \emph{Inverse mean curvature flow inside a cone in warped
products}, submitted and available online at arXiv:1705.04865v3.

\bibitem{cg} C. Gerhardt, \emph{Flow of nonconvex hypersurfaces into
spheres}, J. Differential Geom. {\bf 32} (1990) 299--314.

\bibitem{hkl} C.-L. He, D.-X. Kong and K.-F. Liu, \emph{Hyperbolic mean curvature
flow}, J. Differential Equat. {\bf 246} (2009) 373--390.

\bibitem{gh} G. Huisken, \emph{Flow by mean curvature of convex surfaces into
spheres}, J. Differential Geom. {\bf20} (1984) 237--266.

\bibitem{hi} G. Huisken and T. Ilmanen, \emph{The inverse mean curvature flow and the Riemannian Penrose
inequality}, J. Differential Geom. {\bf 59} (2001) 353--437.

\bibitem{klw} D.-X. Kong, K.-F. Liu and Z.-G. Wang, \emph{Hyperbolic mean curvature flow: evolution of plane
curves}, Acta Math. Scientia {\bf 29(B)}(3) (2009) 493--514.

\bibitem{m1} J. Mao, \emph{Forced hyperbolic mean curvature flow}, Kodai Math.
J. {\bf 35} (2012) 500--522.

\bibitem{wzm} J. Mao, Q. Ming, C.-X. Wu and Z. Zhou, \emph{Hyperbolic inverse curvature flows in warped
products}, preprint.

\bibitem{Ma1} T. Marquardt, \emph{Inverse mean curvature flow for star-shaped hypersurfaces evolving in a
cone}, J. Geom. Anal. {\bf 23} (2013) 1303--1313.

\bibitem{pw} M.-H. Protter, H.-F. Weinberger, \emph{Maximum Principles in Differential Equations},
Springer-Verlag, New York, 1984.

\bibitem{RSC} R. Schneider, \emph{Convex Bodies: The Brum-Minkowski Theory}, Cambridge University
Press, Cambridge, 1993.

\bibitem{pt} P. Topping, \emph{Mean curvature floww and geometric
inequalities}, J. reine angew. Math. {\bf 503} (1998) 47--61.


\bibitem{y1} S.-T. Yau, \emph{Review of geometry and analysis},
Asian J. Math. {\bf 4} (2000) 235--278.

\bibitem{zxp} X.-P. Zhu, \emph{Lectures on Mean Curvature Flows},
Stud. Adv. Math., Vol. 32, AMS/IP, 2002.














\end{thebibliography}
 \end{document}